\documentclass{amsart}

\usepackage{mathrsfs}
\usepackage{dutchcal}
\usepackage{amssymb}
\usepackage{hyperref}
\hypersetup{
	colorlinks=true,
	linkcolor=blue,
	filecolor=blue,      
	urlcolor=blue,
	citecolor=cyan,
}

\newcommand{\sinc}{\mathrm{sinc}}

\newtheorem{theorem}{Theorem}[section]
\newtheorem{lemma}[theorem]{Lemma}

\theoremstyle{definition}
\newtheorem{definition}[theorem]{Definition}

\newtheorem{prop}[theorem]{Proposition}

\theoremstyle{remark}
\newtheorem{remark}[theorem]{Remark}

\numberwithin{equation}{section}

\begin{document}
	
	\title{Paley-Wiener Theorems For Slice Regular Functions
	}
	
	%    Remove any unused author tags.
	
	%    author one information
	\author{Yanshuai Hao}
	\address{School of Computer Science and Engineering, Faculty of Innovation Engineering, Macau University of Science and Technology}
	\curraddr{}
	\email{3220004726@student.must.edu.mo}
	\thanks{}
	
	%    author two information
	\author{Pei Dang}
	\address{Department of Engneering Science, Faculty of Innovation Engineering, Macau University of Science and Technology}
	\curraddr{}
	\email{pdang@must.edu.mo}
	\thanks{}
	
	%    author three information
	\author{Weixiong Mai$^{\star}$}
	\address{School of Computer Science and Engineering, Faculty of Innovation Engineering, Macau University of Science and Technology}
	\curraddr{}
	\email{wxmai@must.edu.mo}
	\thanks{$^{\star}$Corresponding author}
	
	%\subjclass[2010]{Primary}
	
	
	\date{}
	
	\dedicatory{}
	
	\begin{abstract}
		We prove two theorems of Paley and Wiener in the slice regular setting. As an application, we can compute the reproducing kernel for the slice regular Paley-Wiener space, and obtain a related sampling theorem.
	\end{abstract}
	
	\maketitle
	
	\section{Introduction}
	In complex analysis, the classical Paley-Wiener theorem for holomorphic functions usually has two different forms (see \cite{paley1934fourier}). One is about entire functions of exponential type, and the other is about the Hardy $H^2$ space on the upper half-plane. The first one states that an entire function $f$ satisfies $f\left(z\right)=O\left(e^{A|z|}\right)$ with $f|_\mathbb R \in L^2\left(\mathbb{R}\right)$ if and only if $\mathscr{F}\left(f|_{\mathbb{R}}\right)$ is supported in $[-A,A]$, where $f|_{\mathbb{R}}$ is the restriction of $f$ on $\mathbb R$, and  $\mathscr{F}\left(f|_{\mathbb{R}}\right)$ is the Fourier transform of $f|_{\mathbb{R}}$. The other one states that if $f$ is a function in $L^2(\mathbb R)$, then $f$ is the non-tangential boundary limit (NTBL) of some function in the Hardy $H^2$ space  on the upper half-plane if and only if the support of $\mathscr{F}\left(f\right)$ lies on $[0,\infty).$  For simplicity, in this paper we call them the compact type and the non-compact type Paley-Wiener theorems, respectively. 
	
	Due to its importance to the study of function theory in complex analysis, many studies about Paley-Wiener theorem are given in the literature. For instance, there exist analogous results of Paley-Wiener theorem in $L^p(\mathbb R)$. For the non-compact case, one has that for $f\in L^p(\mathbb R)$, $1\leq p\leq\infty$, $f$ is the NTBL of some function in the Hardy $H^p$ space on the upper half-plane if and only if $\left(\mathscr{F}\left(f\right),\varphi\right)=0$ for $\varphi$ in Schwartz space with $\operatorname{supp}\varphi\subset(-\infty,0]$ (see \cite{qian2005characterization,qian2009fourier}). For the compact case, in \cite{plancherel1937fonctions} it was proved that if an entire function $f$ satisfies $f\left(z\right)=O\left(e^{A|z|}\right)$ and $f|_{\mathbb{R}}\in L^{p}(\mathbb R)$ for $1<p\leq 2$, then $f\left(z\right)=\int_{-A}^{A}e^{izt}\mathscr{F}\left(f|_{\mathbb{R}}\right)\left(t\right)\mathrm{d}t$ with $\mathscr{F}\left(f|_{\mathbb{R}}\right)\in L^{\frac{p}{p-1}}(-A,A)$, and conversely, in \cite{boas1938representations} it was shown that if $f\left(z\right)=\int_{-A}^{A}e^{izt}\mathscr{F}\left(f|_{\mathbb{R}}\right)\left(t\right)\mathrm{d}t$ with $\mathscr{F}\left(f|_{\mathbb{R}}\right)\in L^{p}(-A,A)$, then $f$ is an entire function satisfying $f\left(z\right)=O\left(e^{A|z|}\right)$ and $f|_{\mathbb{R}}\in L^{\frac{p}{p-1}}(\mathbb R)$. 
	
	Besides, the study of Paley-Wiener theorem is also considered in several variables. In the several complex variables setting, the non-compact type Paley-Wiener theorem states that $f$ is the NTBL of some function in  the Hardy spaces $H^2\left(T_\Gamma\right)$ on tubes $T_\Gamma$ over regular cones $\Gamma$ if and only if $\operatorname{supp}\mathscr{F}\left(f\right)\subset\Gamma^{*}$, where $\Gamma^{*}$ is the dual cone of $\Gamma$ (see e.g. \cite{stein1971introduction}). In recent years, the analogous results are also obtained in $H^p(T_\Gamma)$ for $1\leq p\leq \infty$ and $0<p<1$, respectively (see \cite{li2018fourier} and \cite{deng_fourier_2019}), while a weighted version of the Paley-Wiener type theorem in a tubular domain over a regular cone is investigated in \cite{deng2019paley}. As an analogue of the compact case, the Paley-Wiener theorem is generalized to entire functions of several complex variables with the exponential type bounds with respect to the setting of symmetric body (see e.g. \cite{stein1971introduction}). In the Clifford algebra setting, the compact type Paley-Wiener theorem is obtained in \cite{kou2002paley}. 
	The non-compact type Paley-Wiener theorem is respectively studied in \cite{bernstein1998paley} and \cite{gilbert1991clifford} for $p=2,$ while for $1\leq p\leq \infty,$ a systematical investigation of Paley-Wiener theorem is given in \cite{dang2020fourier}. Furthermore, the Paley-Wiener type theorems have been extensively studied including generalizations in the distribution sense, analogous results in Bergman spaces and Dirichlet spaces (see e.g. \cite{GENCHEV1986496,bernstein1998paley,duren2007paley,qian2005characterization,gilbert1991clifford,qian2009fourier,li2018fourier,dang2020fourier,os2001generalized,hormander2015analysis,schwartz_transformation_1952}).
	
	In the past two decades, the slice regular function theory as a new analytic function theory has been developed well in high dimensional space, which is different from the settings of holomorphic functions in several complex variables and monogenic functions in Clifford algebras. It is noted that a lot of fundamental properties and theorems of analytic functions in complex analysis have been generalized to slice regular functions (see e.g. \cite{colombo2016entire,gentili2009open,gentili2012power,gentili2011phragmen,gentili2007new,colombo2010cauchy} and the references therein). It is well-known that the classical Paley-Wiener theorem plays an important role in complex analysis.  So it would be interesting and significant to consider this theorem for slice regular functions. To the authors' best knowledge, a generalization of Paley-Wiener theorem is not fully obtained in such setting. In \cite{LiangPaley2013}, the authors considered a kind of compact type Paley-Wiener theorem in a particular slice with respect to linear canonical transform, which is fairly obtained from the arguments in complex analysis (see e.g. \cite{rudin1987real}). The problem is that their result can not be extended to the whole space. 
	
	In this paper we will prove the compact and non-compact types of Paley-Wiener theorem for slice regular functions. To state the main results, we introduce some notations. Denote by $\mathbb H$ the quaternions space, by $\mathscr R_l(U)$ the class of left slice regular functions on the open set $U\subset \mathbb H,$ by $\mathbb H_+$ the open right half-space of $\mathbb H,$  by $H^p(\mathbb H_+)$ the Hardy space on $\mathbb H_+,$ by $\mathscr{F}_{I}$ the one-dimensional left-sided quaternion Fourier transform (see Definition \ref{def-1DQFT}), by $\mathscr{F}_{E}$ the essential Fourier transform (see Definition \ref{essen_fourier}), and by $\mathbb{S}$ the unit sphere of purely imaginary quaternions. 
	
	The compact type Paley-Wiener theorem is stated as follows. The main trick to prove this theorem is to use a decomposition of $\mathscr R_l(U)$ and extension of slice functions (see Proposition \ref{prop-SRE} and Proposition \ref{prop-NE}).
	
	\smallskip
	
	\noindent \textbf{Main Result I: Compact type.}
	{\it
		Let $f:\mathbb{H}\to\mathbb{H}$ be a left slice regular function with $f|_\mathbb{R}\in L^2(\mathbb{R})$, i.e.,
		\begin{equation*}
			\int_{-\infty}^{+\infty}|f(x)|^2\mathrm{d}x<\infty,
		\end{equation*}
		and $A>0$ be a positive number. Then the following two conditions are equivalent:
		\begin{enumerate}
			\item[(i)] There exists a constant $C$ such that 
			\begin{equation*}
				|f(\mathbf{q})|\leq Ce^{A|\mathbf{q}|}
			\end{equation*}
			for all $\mathbf{q}\in\mathbb{H}$.
			\item[(ii)] $\operatorname{supp}\mathscr{F}_{I}\left(f|_{\mathbb{R}}\right)\subset\left[-A,A\right]$ for all $I\in\mathbb{S}$.
		\end{enumerate}
		Moreover, if one of the above conditions holds, then we have 
		\begin{equation*}
			f(x+Iy)=\frac{1}{\sqrt{2\pi}}\int_{-A}^{A}e^{I\left(x+Iy\right)t}\mathscr{F}_{I}\left(f|_{\mathbb{R}}\right)(t)\mathrm{d}t
		\end{equation*}for all $I\in\mathbb{S}$ and $x,y\in\mathbb{R}$.
	}
	
	\smallskip 
	\smallskip
	
	For the case of non-compact type, we can prove the following result by a similar trick. 
	
	\noindent\textbf{Main Result II: Necessary condition for non-compact type.}
	{\it
		Let $f\in H^{p}\left(\mathbb{H}_{+}\right)$, $1\leq p\leq2$. Then $\operatorname{supp}\mathscr{F}_{E}\left(f\right)\subset\left(-\infty,0\right]$. Moreover, for every $I\in\mathbb{S}$ and $x>0,y\in\mathbb{R}$,  
		\begin{equation*}
			f\left(x+Iy\right)=\frac{1}{\sqrt{2\pi}}\int_{-\infty}^{0}e^{\left(x+Iy\right)t}\mathscr{F}_{E}\left(f\right)\left(t\right)\mathrm{d}t.
		\end{equation*}
    }
	
	\noindent\textbf{Main Result II: Sufficient condition for non-compact type.}
	{\it
		Let $F:\mathbb{R}\to\mathbb{H}$ and  $F\in L^{p}(\mathbb{R})$ for $1\leq p\leq2$. Assume that there exists a quaternionic imaginary unit $I\in\mathbb{S}$ such that $\operatorname{supp}\mathscr{F}_{I}\left(F\right)\subset\left(-\infty,0\right]$. Then there exists a function $f\in H^p(\mathbb{H}_+)$, whose non tangential boundary boundary limit at $Iy$ is $F(y)$, whence $\mathscr{F}_{E}\left(f\right)=\mathscr{F}_{I}\left(F\right)$ and
        \begin{equation*}
			f\left(x+Jy\right):=\frac{1}{\sqrt{2\pi}}\int_{-\infty}^{0}e^{\left(x+Jy\right)t}\mathscr{F}_{I}\left(F\right)\left(t\right)\mathrm{d}t,
		\end{equation*}
        for all $x>0,y\in\mathbb{R}$, $J\in\mathbb{S}$. 

        }
	
	As an application of \textbf{Main Result I}, we obtain the reproducing kernel of slice regular Paley-Wiener space, and prove the related sampling theorem.

	\smallskip 
	\smallskip
	
	The paper is organized as follows. In \S 2 we introduce some notations for quaternions and slice regular functions. In \S 3 we prove the main results. In \S 4 we compute specifically the reproducing kernel of slice regular Paley-Wiener space and prove the related sampling theorem.   
	\section{Preliminaries}
	
	In this section we briefly introduce some notations and properties for slice regular functions. We refer to \cite{gentili2013regular,alpay2016slice,colombo2016entire} for more details.
	The quaternion space $\mathbb{H}$ is defined as
	\[\mathbb{H}=\left\{\mathbf{q}=x_0+\mathbf{i}x_1+\mathbf{j}x_2+\mathbf{k}x_3: x_0,x_1,x_2,x_3\in\mathbb{R}\right\},\]
	where $\mathbf{i}, \mathbf{j}$ and $ \mathbf{k}$ satisfy
	\[\mathbf{i}^2=\mathbf{j}^2=\mathbf{k}^2=-1,\mathbf{i}\mathbf{j}=-\mathbf{j}\mathbf{i}=\mathbf{k},\mathbf{j}\mathbf{k}=-\mathbf{k}\mathbf{j}=\mathbf{i},\mathbf{k}\mathbf{i}=-\mathbf{i}\mathbf{k}=\mathbf{j}.\]
	The conjugation of $\mathbf{q}\in \mathbb H$ is given by
	\[\overline{\mathbf{q}}=x_0-\mathbf{i}x_1-\mathbf{j}x_2-\mathbf{k}x_3,\]
	and the norm of $\mathbf{q}$ is 
	\[|\mathbf{q}|=\sqrt{\mathbf{q}\overline{\mathbf{q}}}=\sqrt{x_0^2+x_1^2+x_2^2+x_3^2}.\]
	The symbol $\mathbb{S}$ denotes the unit sphere of purely imaginary quaternion, i.e.,
	\[\mathbb{S}=\left\{\mathbf{q}=\mathbf{i}x_1+\mathbf{j}x_2+\mathbf{k}x_3: x_1^2+x_2^2+x_3^2=1,x_1,x_2,x_3\in\mathbb{R}\right\}.\]
	Note that if $I\in\mathbb{S}$, then $I^2=-1$. For any fixed $I\in\mathbb{S}$ we let
	\[\mathbb{C}_{I}:=\left\{x+Iy:x,y\in\mathbb{R}\right\}.\]
	One can verify that $\mathbb{C}_{I}$ can be identified with a complex plane, moreover $\mathbb{H}=\bigcup_{I\in\mathbb{S}}\mathbb{C}_{I}$. For every $\mathbf{q}\in\mathbb{H}$, we can write $\mathbf{q}=x_{0}+I_{\mathbf{q}}y_{0}$ where $x_{0}=\frac{\mathbf{q}+\overline{\mathbf{q}}}{2}$, $y_{0}=|\mathbf{q}-\overline{\mathbf{q}}|$ and $I_\mathbf{q}=\frac{\mathbf{q}-\overline{\mathbf{q}}}{|\mathbf{q}-\overline{\mathbf{q}}|}$. It is noted that $\mathbf{q}$ belongs to the complex plane $\mathbb{C}_{I_\mathbf{q}}$.

	\begin{definition}[\cite{colombo2016entire}]\label{def-SR}
		Let $\Omega$ be an open set in $\mathbb{H}$ and $f:\Omega\to\mathbb{H}$ be real differentiable. The function $f$ is said to be left slice regular if for every $I\in\mathbb{S}$, its restriction $f_I$ to the complex plane $\mathbb{C}_{I}$ satisfies
		\[\overline{\partial_I}f\left(x+Iy\right):=\frac{1}{2}\left(\frac{\partial}{\partial x}+I\frac{\partial}{\partial y}\right)f_I\left(x+Iy\right)=0\]
		on $\Omega\cap\mathbb{C}_{I}$. The class of left slice regular functions on $\Omega$ will be denoted by $\mathscr{R}_{l}(\Omega)$. Similarly, a function is said to be right slice regular in $\Omega$ if
		\begin{equation*}
			(f_I\overline{\partial_I})\left(x+Iy\right):=\frac{1}{2}\left(\frac{\partial}{\partial x}f_I\left(x+Iy\right)+\frac{\partial}{\partial y}f_I\left(x+Iy\right)I\right)=0
		\end{equation*}
		on $\Omega\cap\mathbb{C}_{I}$.
	\end{definition}
	\noindent In fact, we note that Definition \ref{def-SR} is a re-elaboration of a definition given by Gentili, Stoppato and Struppa (see \cite[Chapter 1]{gentili2013regular}).

	In this paper, we only discuss the case for left slice regular functions. Note that for the case of right slice regular functions we can have similar results.
	
		\begin{definition}[\cite{gentili2013regular}]
		    Let $\Omega$ be an open set in $\mathbb{H}$ and $f\in\mathscr{R}_{l}(\Omega)$. If $f$ satisfies $f(\Omega\cap\mathbb{C}_{I})\subseteq\mathbb{C}_{I}$ for all $I\in\mathbb{S}$, then $f$ is called left slice preserving. The class of left slice preserving functions on $\Omega$ will be denoted by $\mathscr{N}_{l}(\Omega)$.
        \end{definition}

    \begin{prop}[\cite{gentili2007new}]\label{prop-series}
        Let $B(0,R)$ be a ball centered at the origin and of radius $R$. A function $f\in\mathscr{R}_{l}(B(0,R))$ if, and only if, it has a series expansion of the form
        \begin{equation*}
            f(\mathbf{q})=\sum_{n=0}^{\infty}\mathbf{q}^n \frac{1}{n!}\frac{\partial^n f}{\partial x^n}(0)
        \end{equation*}
        converging on $B(0,R)$. In particular if $f\in\mathscr{R}_{l}(B(0,R))$ then it is $C^{\infty}$ on $B(0,R)$. 
    \end{prop}

    \begin{prop}[\cite{gentili2013regular,gentili2007new}]\label{prop-N-series}
         Let $f\in\mathscr{R}_{l}(B(0,R))$ on the ball $B(0,R)$ centered at the origin with radius $R$, and let $$f(\mathbf{q})=\sum_{n=0}^{\infty}\mathbf{q}^n a_n$$ be power series expansion. Then $f\in\mathscr{N}_{l}(B(0,R))$ if, and only if, $\{a_n\}_{n\in\mathbb{N}}\subset\mathbb{R}$.
    \end{prop}

	%	\begin{prop}[\cite{colombo2016entire}, Splitting Lemma]\label{prop-SL}
		%		If $f$ is a left slice regular function on $U$, then for every $I\in\mathbb{S}$, and every $J\in\mathbb{S}$, $J\perp I$, there are two holomorphic functions $F,G:U\cap\mathbb{C}_{I}\to\mathbb{C}_{I}$ such that for any $z=x+Iy,$
		%		\[f_I(z)=F(z)+\Psi\left(z\right)J.\]
		%	\end{prop}
	\begin{prop}[\cite{colombo2009extension,ghiloni2013continous}  Representation Formula]\label{prop-RF}
		Let $f$ be a left slice regular function defined on an axially symmetric slice domain $U\subseteq\mathbb{H}$. Let $J\in\mathbb{S}$ and  $x\pm Jy\in U \cap\mathbb{C}_J$. Then the following equality holds for all $\mathbf{q}=x+Iy\in U$:
		\begin{align*}
			f(x+Iy)&=\frac{1}{2}\left[f\left(x-Jy\right)+f\left(x+Jy\right)\right]+\frac{1}{2}IJ\left[f\left(x-Jy\right)-f\left(x+Jy\right)\right]\notag\\
			&=\frac{1}{2}\left(1-IJ\right)f\left(x+Jy\right)+\frac{1}{2}\left(1+IJ\right)f\left(x-Jy\right).
		\end{align*}
		Moreover, for all $x+Ky\subseteq U$, $K\in\mathbb{S}$, there exist two functions $\alpha,\beta$, independent
		of $I$, such for any $K\in\mathbb{S}$ we have
		\begin{equation*}
			\frac{1}{2}\left[f\left(x-Ky\right)+f\left(x+Ky\right)\right]=\alpha(x,y)
		\end{equation*}
		and
		\begin{equation*}
			\frac{1}{2}K\left[f\left(x-Ky\right)-f\left(x+Ky\right)\right]=\beta(x,y).
		\end{equation*}
	\end{prop}
	As a consequence of Proposition \ref{prop-RF}, left slice regular functions have the following property.
 \begin{prop}[\cite{gentili2013regular,gentili2011weierstrass}]\label{prop-fIinBR-finB2R}
     Let $U\subseteq\mathbb{H}$ be an axially symmetric slice domain and $f\in\mathscr{R}_{l}(U)$. Let $V\subseteq U$ be a symmetric compact set. For every $I\in\mathbb{S}$, $\mathbf{q}\in\mathbb{H}$, and $R>0$ such
that
\begin{equation*}
f_{I}\left(V_{I}\right)\subseteq B\left(\mathbf{q},R\right)
\end{equation*}
the inclusion
\begin{equation*}
    f\left(V\right)\subseteq B\left(\mathbf{q},2R\right)
\end{equation*}
holds.
 \end{prop}
 
	\begin{prop}[\cite{ghiloni2011slice}]\label{prop-f=a+Ib}
		Let $U\subseteq\mathbb{H}$ be an axially symmetric slice domain, let $D\subseteq\mathbb{R}^2$ be such
		that $x+Iy\in U$ whenever $(x,y)\in D$ and let $f:U\to\mathbb{H}$. The function $f$ is left slice
		regular if and only if there exist two differentiable functions $\alpha,\beta:D\subseteq\mathbb{R}^2\to\mathbb{H}$ satisfying $\alpha(x,y)=\alpha(x,-y)$, $\beta(x,y)=-\beta(x,-y)$ and the Cauchy–Riemann system
		\begin{equation*}
			\begin{cases}
				\partial_{x}\alpha-\partial_{y}\beta=0\\
				\partial_{y}\alpha+\partial_{x}\beta=0
			\end{cases}
		\end{equation*}
		such that $f\left(x+Iy\right)=\alpha(x,y)+I\beta(x,y)$.
	\end{prop}
	
		Proposition \ref{prop-f=a+Ib} and Proposition \ref{prop-RF} imply that $f\in\mathscr{N}_{l}(U)$ if, and only if 
        \[\alpha\left(x,y\right)=\frac{1}{2}\left[f\left(x+Iy\right)+f\left(x-Iy\right)\right]\]
        and
        \[\beta\left(x,y\right)=\frac{-I}{2}\left[f\left(x+Iy\right)-f\left(x-Iy\right)\right]\]
        are two real-valued functions (see \cite[Proposition 10]{ghiloni2011slice} and \cite[Proposition 2.5]{colombo2016entire}). if this is the case and if $x+Iy \in U$, it follows at once that $\Re f\left(x+Iy\right) = \alpha\left(x,y\right)$ and that $\lvert f\left(x+Iy\right)\rvert=\sqrt{\alpha\left(x,y\right)^{2}+\beta\left(x,y\right)^{2}}$, whence $\lvert f\left(x+Iy\right)\rvert=\lvert f\left(x+Jy\right)\rvert$ for all $I,J\in\mathbb{S}$. Moreover, the following result is given.
	
	\begin{prop}[\cite{colombo2016entire}]\label{prop-N_l}
		Let $U\subseteq\mathbb{H}$ be an axially symmetric slice domain. Then $f\in\mathscr{N}_{l}(U)$ if and only if $f\in\mathscr{R}_{l}(U)$ satisfies $f(\overline{\mathbf{q}})=\overline{f(\mathbf{q})}$ for all $\mathbf{q}\in U$.
	\end{prop}
 	
	\begin{prop}[\cite{colombo2009extension}]\label{prop-SRE}
		Let $U_J$ be a domain in $\mathbb{C}_J$ symmetric with respect to the real axis and such that $U_J\cap\mathbb{R}\neq\varnothing$. Let $U$ be the axially symmetric slice domain defined by
		\[U=\bigcup_{x+Jy\in U_{J},I\in\mathbb{S}}\left\{x+Iy\right\}.\]
		If $f:U_J\to\mathbb{H}$ satisfies $\overline{\partial_J}f=0$ then the function
		\[\mathbf{ext}_{l}(f)\left(x+Iy\right):=\frac{1}{2}\left[f\left(x-Jy\right)+f\left(x+Jy\right)\right]+\frac{1}{2}IJ\left[f\left(x-Jy\right)-f\left(x+Jy\right)\right]\]
		is the unique left slice regular extension of $f$ to $U$.
	\end{prop}  
 Furthermore, if $f:U_{J}\to\mathbb{C}_{J}$ satisfies $\overline{\partial_J}f=0$ and $f(x-Jy)=\overline{f(x+Jy)}$, then $\mathbf{ext}_{l}(f)\in\mathscr{N}_{l}(U)$, in fact we have:
 \begin{prop}\label{prop-extl(f)inN}
     Let $U_J$ be a domain in $\mathbb{C}_J$ symmetric with respect to the real axis and such that $U_J\cap\mathbb{R}\neq\varnothing$. Let $U$ be the axially symmetric slice domain defined by
		\[U=\bigcup_{x+Jy\in U_{J},I\in\mathbb{S}}\left\{x+Iy\right\}.\]
		If $f:U_J\to\mathbb{C}_{J}$ satisfies $\overline{\partial_J}f=0$ and $f(x-Jy)=\overline{f(x+Jy)}$ then $\mathbf{ext}_{l}(f)\in\mathscr{N}_{l}(U)$.
 \end{prop}
 \begin{proof}
     Since $f$ is $\mathbb{C}_{J}$-valued and satisfies $f(x-Jy)=\overline{f(x+Jy)}$ for all $x+Jy\in U_{J}$, we have that 
     \[\alpha\left(x,y\right)=\frac{1}{2}\left[f(x+Jy)+f(x-Jy)\right]\]
     and
     \[\beta\left(x,y\right)=\frac{-J}{2}\left[f(x+Jy)-f(x-Jy)\right]\]
     are two real-valued functions. Since $U$ is the axially symmetric slice domain, we have that if $x+Jy\in U_{J}$ then $x+Iy\in U$ for all $I\in\mathbb{S}$. Since $\overline{\partial_{J}}f=0$, using Proposition \ref{prop-SRE}, we can get that for all $x+Jy\in U_{J}$ and $I\in\mathbb{S}$, 
     $$\mathbf{ext}_{l}(f)(x+Iy)=\alpha(x,y)+I\beta(x,y)$$ is in $\mathscr{R}_{l}(U)$. Note that 
     \begin{align*}
         \mathbf{ext}_{l}(f)\left(x-Iy\right)=&\alpha\left(x,-y\right)+I\beta\left(x,-y\right)\\
         =&\alpha\left(x,y\right)-I\beta\left(x,y\right)\\
         =&\overline{\alpha\left(x,y\right)+I\beta\left(x,y\right)}\\
         =&\overline{\mathbf{ext}_{l}(f)\left(x+Iy\right)}.
     \end{align*}
     Hence, $\mathbf{ext}_{l}(f)\in\mathscr{N}_{l}(U)$.
 \end{proof}
 
	\begin{prop}[\cite{colombo2016entire,ghiloni2013continous}]\label{prop-NE}
			Let $U\subseteq\mathbb{H}$ be an axially symmetric slice domain and $\left\{1,I,J,IJ\right\}$ be a basis of $\mathbb{H}$. Then
			\begin{equation*}
				\mathscr{R}_{l}(U)=\mathscr{N}_{l}(U)\oplus\mathscr{N}_{l}(U)I\oplus\mathscr{N}_{l}(U)J\oplus\mathscr{N}_{l}(U)IJ.
			\end{equation*}
	\end{prop}
	\begin{definition}[\cite{ell2014quaternion}]\label{def-1DQFT}%1D quaternion Fourier transforms
			Let $F:\mathbb{R}\to\mathbb{H}$ be a quaternion-valued function with $F\in L^1(\mathbb{R})$ and $I\in\mathbb{S}$. The one-dimensional left-sided quaternion Fourier transform $\mathscr{F}_I(F)$ of $F$ is defined by
			\begin{equation*}
				\mathscr{F}_{I}\left(F\right)(t):=\frac{1}{\sqrt{2\pi}}\int_{\mathbb{R}}e^{-Ixt}F(x)\mathrm{d}x.
		\end{equation*}	
	\end{definition}
		 
		\begin{prop}\label{prop-FTf|R}
			Let $f\in\mathscr{N}_{l}(\mathbb{H})$ and $f|_{\mathbb{R}}\in L^{1}(\mathbb{R})$. Then for every $I,J\in\mathbb{S}$ and $t\in\mathbb{R}$, we have
			\begin{enumerate}
                \item $\mathscr{F}_{I}\left(f|_{\mathbb{R}}\right)(t)$ is $\mathbb{C}_{I}$-valued,
				\item $\overline{\mathscr{F}_{I}\left(f|_{\mathbb{R}}\right)(t)}=\mathscr{F}_{I}\left(f|_{\mathbb{R}}\right)(-t)$,
				\item $\mathscr{F}_{J}\left(f|_{\mathbb{R}}\right)(t)=\frac{1}{2}\left[\mathscr{F}_{I}\left(f|_{\mathbb{R}}\right)(t)+\mathscr{F}_{I}\left(f|_{\mathbb{R}}\right)(-t)\right]+\frac{JI}{2}\left[\mathscr{F}_{I}\left(f|_{\mathbb{R}}\right)(-t)-\mathscr{F}_{I}\left(f|_{\mathbb{R}}\right)(t)\right]$.
			\end{enumerate}
			%where \[\Re\left\{\mathscr{F}_{I}\left(f|_{\mathbb{R}}\right)(t)\right\}=\frac{1}{2}\left(\mathscr{F}_{I}\left(f|_{\mathbb{R}}\right)(t)+\overline{\mathscr{F}_{I}\left(f|_{\mathbb{R}}\right)(t)}\right)\]
			%and
			%\[\Im\left\{\mathscr{F}_{I}\left(f|_{\mathbb{R}}\right)(t)\right\}=\frac{-I}{2}\left(\mathscr{F}_{I}\left(f|_{\mathbb{R}}\right)(t)-\overline{\mathscr{F}_{I}\left(f|_{\mathbb{R}}\right)(t)}\right).\]
		\end{prop}
		\begin{proof}
			Since $f\in\mathscr{N}_{l}(\mathbb{H})$, we have $f(\overline{\mathbf{q}})=\overline{f(\mathbf{q})}$ for all $\mathbf{q}\in\mathbb{H}$. Hence, $f(x)=\overline{f(x)}$ for all $x\in\mathbb{R}$, i.e., $f|_{\mathbb{R}}$ is real-valued. By $f|_{\mathbb{R}}\in L^{1}(\mathbb{R})$, for $I\in\mathbb{S}$ and $t\in\mathbb{R}$, we have
			\[\mathscr{F}_{I}\left(f|_{\mathbb{R}}\right)(t)=\frac{1}{\sqrt{2\pi}}\int_{\mathbb{R}}e^{-Ixt}f(x)\mathrm{d}x.\]
            It is easy to verify that $\mathscr{F}_{I}\left(f|_{\mathbb{R}}\right)(t)$ is $\mathbb{C}_{I}$-valued and $\overline{\mathscr{F}_{I}\left(f|_{\mathbb{R}}\right)(t)}=\mathscr{F}_{I}\left(f|_{\mathbb{R}}\right)(-t)$.
			Note that for any fixed $t\in\mathbb{R}$, $e^{-Ixt}=\cos(xt)-I\sin(xt)$
            holds for all $x\in\mathbb{R.}$ Hence, we have that
			\[\frac{1}{2}\left[\mathscr{F}_{I}\left(f|_{\mathbb{R}}\right)(t)+\mathscr{F}_{I}\left(f|_{\mathbb{R}}\right)(-t)\right]=\frac{1}{\sqrt{2\pi}}\int_{\mathbb{R}}\cos(xt)f(x)\mathrm{d}x\]
			and
			\[\frac{-I}{2}\left[\mathscr{F}_{I}\left(f|_{\mathbb{R}}\right)(t)-\mathscr{F}_{I}\left(f|_{\mathbb{R}}\right)(-t)\right]=-\frac{1}{\sqrt{2\pi}}\int_{\mathbb{R}}\sin(xt)f(x)\mathrm{d}x\]
            are real-valued.
            Then, for $J\in\mathbb{S}$, we have
            \begin{align*}
            	\mathscr{F}_{J}\left(f|_{\mathbb{R}}\right)(t)&=\frac{1}{\sqrt{2\pi}}\int_{\mathbb{R}}e^{-Jxt}f(x)\mathrm{d}x\\
            	&=\frac{1}{\sqrt{2\pi}}\int_{\mathbb{R}}\cos(xt)f(x)\mathrm{d}x-J\frac{1}{\sqrt{2\pi}}\int_{\mathbb{R}}\sin(xt)f(x)\mathrm{d}x\\
            	&=\frac{1}{2}\left[\mathscr{F}_{I}\left(f|_{\mathbb{R}}\right)(t)+\mathscr{F}_{I}\left(f|_{\mathbb{R}}\right)(-t)\right]+\frac{JI}{2}\left[\mathscr{F}_{I}\left(f|_{\mathbb{R}}\right)(-t)-\mathscr{F}_{I}\left(f|_{\mathbb{R}}\right)(t)\right].
            \end{align*}
		\end{proof}

        It is easy to show that the conclusion of Proposition \ref{prop-FTf|R} also holds for $f\in\mathscr{N}_{l}(\mathbb{H})$ with $f|_{\mathbb{R}}\in L^{2}(\mathbb{R}).$ Moreover,
		Proposition \ref{prop-FTf|R} will be used to prove Lemma \ref{PW-NQ}.
		\begin{definition}[\cite{de2018quaternionic,sarfatti2013elements}]
			The open right half-space of the quaternions is denoted by $\mathbb{H}_+=\left\{\mathbf{q}=x+Iy:x>0,y\in\mathbb{R},I\in\mathbb{S}\right\}$. We set $\Pi_{+,I}=\mathbb{H}_{+}\cap \mathbb{C}_{I}$. For $1\leq p\leq 2$, we define
			\begin{equation*}
				H^{p}\left(\Pi_{+,I}\right)=\left\{f\in\mathscr{R}_{l}\left(\mathbb{H}_{+}\right);\int_{-\infty}^{+\infty}\lvert F^{I}\left(y\right)\rvert^{p}\mathrm{d}y<\infty\right\},
			\end{equation*}
			where $F^{I}\left(y\right)=f\left(Iy\right)$ denotes the non-tangential value of $f$ at $Iy$. We define $H^p\left(\mathbb{H}_+\right)$ as the space of functions $f\in\mathscr{R}_l\left(\mathbb{H}_+\right)$ such that
			\[\sup_{I\in\mathbb{S}}\int_{-\infty}^{+\infty}|F^{I}\left(y\right)|^p\mathrm{d}y<\infty.\]
	\end{definition}
  \begin{remark}
       If $f\in H^{p}\left(\mathbb{H}_{+}\right)$, $1\leq p\leq2$, then for all $x+Iy\in\mathbb{H}_{+}$ and $J\in\mathbb{S}$, we have 
      \begin{align*}
              f\left(x+Jy\right)&=\frac{1}{2}\left[f\left(x+Iy\right)+f\left(x-Iy\right)\right]+\frac{JI}{2}\left[f\left(x-Iy\right)-f\left(x+Iy\right)\right]\\
          &=\alpha(x,y)+J\beta(x,y),
      \end{align*} 
      where $\alpha$ and $\beta$ are independent of $J$.
      Let $F^{J}\left(y\right)=f\left(Jy\right)$ be the NTBL of $f$ at $Jy$. Then $F^{J}\left(y\right)=\alpha\left(0,y\right)+J\beta\left(0,y\right)$ holds a. e. $y\in\mathbb{R}$. It means that $F^{J}$ depends on $J$. In particular, if $f\in H^{p}\left(\mathbb{H}_{+}\right)\cap\mathscr{N}_{l}(\mathbb{H}_{+})$, we can obtain that $\alpha(0,y)$ and $\beta(0,y)$ are two real-valued functions.
  \end{remark}
    	\begin{prop}\label{prop-FTf_I^+}
    		Let $f\in\mathscr{N}_{l}(\mathbb{H}_{+})\cap H^{1}(\mathbb{H}_{+})$. Then for all $I,J\in\mathbb{S}$, the function $\mathscr{F}_{I}(F^{I})$ is real-valued and  coincides with $\mathscr{F}_{J}(F^{J})$, i.e.,
            \begin{enumerate}
                \item $\overline{\mathscr{F}_{I}(F^{I})(t)}=\mathscr{F}_{I}(F^{I})(t)$;
                \item $\mathscr{F}_{J}(F^{J})(t)=\mathscr{F}_{I}(F^{I})(t).$
            \end{enumerate}
    	\end{prop}
          
    	\begin{proof}
    		Since $f\in\mathscr{N}_{l}(\mathbb{H}_{+})\cap H^1(\mathbb{H}_+)$, for every $I\in\mathbb{S}$, we obtain that $F^{I}\in L^{1}(\mathbb{R})$ and 
            \[F^{I}(y)=\alpha(0,y)+I\beta(0,y),\]
            where $\alpha(0,y)$ and $\beta(0,y)$ are two real-valued functions. It follows that
    		\[\mathscr{F}_{I}(F^{I})(t)=\frac{1}{\sqrt{2\pi}}\int_{\mathbb{R}}e^{-Iyt}F^{I}(y)\mathrm{d}y,\quad t\in\mathbb{R}.\]
    		Note that
    		\begin{align*}
    		    \mathscr{F}_{I}(F^{I})(t)=&\frac{1}{\sqrt{2\pi}}\int_{\mathbb{R}}\left[\cos(yt)\alpha(0,y)+\sin(yt)\beta(0,y)\right]\mathrm{d}y\\
                &+\frac{I}{\sqrt{2\pi}}\int_{\mathbb{R}}\left[\cos(yt)\beta(0,y)-\sin(yt)\alpha(0,y)\right]\mathrm{d}y.
    		\end{align*}
            The second integral vanishes because, for $t$ fixed, the integrand is odd.
    		Hence, for all $J\in\mathbb{S}$,  we have $\mathscr{F}_{J}(F^{J})=\mathscr{F}_{I}(F^{I})$ which is real-valued.
    	\end{proof}
        One can easily show that the conclusion of Proposition \ref{prop-FTf_I^+} also holds for $f\in\mathscr{N}_{l}(\mathbb{H}_{+})\cap H^{2}(\mathbb{H}_{+}).$ Moreover, Proposition \ref{prop-FTf_I^+} is useful to prove Lemma \ref{PW-NH+-HPF}.
        \begin{remark}
            Here, we can show that if $f\in H^{1}(\mathbb{H}_{+})$ and $\tilde{I}\in\mathbb{S}$ then $\mathscr{F}_{\tilde{I}}(F^{\tilde{I}})$ is independent of $\tilde{I}$. Proposition \ref{prop-FTf_I^+} implies that if $f\in\mathscr{N}_{l}(\mathbb{H}_{+})\cap H^{1}(\mathbb{H}_{+})$ and $\tilde{I}\in\mathbb{S}$ then $\mathscr{F}_{\tilde{I}}(F^{\tilde{I}})$ is independent of $\tilde{I}$. By Proposition \ref{prop-NE}, we have that if $f\in\mathscr{R}_{l}(\mathbb{H}_{+})$ and $x+\tilde{I}y\in\mathbb{H}_{+}$ then there exist four slice preserving functions $f_{0},f_{1},f_{2}$ and $f_{3}$ such that
            \[f(x+\tilde{I}y)=f_{0}(x+\tilde{I}y)+f_{1}(x+\tilde{I}y)I+f_{2}(x+\tilde{I}y)J+f_{3}(x+\tilde{I}y)IJ,\]
            where $I,J$ are mutually orthogonal quaternionic imaginary units. It is easy to see that if $f\in H^{1}(\mathbb{H}_{+})$ then $f_{m}\in\mathscr{N}_{l}(\mathbb{H}_{+})\cap H^{1}(\mathbb{H}_{+})$, $0\leq m\leq3$. It follows that
            \[\mathscr{F}_{\tilde{I}}(F^{\tilde{I}})(t)=\mathscr{F}_{\tilde{I}}(F_{0}^{\tilde{I}})(t)+\mathscr{F}_{\tilde{I}}(F_{1}^{\tilde{I}})(t)I+\mathscr{F}_{\tilde{I}}(F_{2}^{\tilde{I}})(t)J+\mathscr{F}_{\tilde{I}}(F_{3}^{\tilde{I}})(t)IJ,\]
            where $F_{m}^{\tilde{I}}(y):=f_{m}(\tilde{I}y)$ are the NTBL of $f_{m}$ at $\tilde{I}y$, and $\mathscr{F}_{\tilde{I}}(F_{m}^{\tilde{I}})$ are real-valued functions and independent of $\tilde{I}$, $0\leq m\leq3$. Then we can get that $\mathscr{F}_{\tilde{I}}(F^{\tilde{I}})$ is independent of $\tilde{I}$. 
        \end{remark}
Hence, 
        in the following, we introduce the essential Fourier transform for $f\in H^{1}(\mathbb{H}_{+}),$ which can be also easily extended to $f\in H^p(\mathbb H_+), 1\leq p\leq 2,$ by the standard argument for the classical Fourier transform.
        \begin{definition}\label{essen_fourier}
        Let $f\in H^{1}(\mathbb{H}_{+})$ and let $\left\{1,I,J,IJ\right\}$ be a basis of $\mathbb{H}_{+}$. It means that there exist four functions $f_{0},f_{1},f_{2},f_{3}\in\mathscr{N}_{l}(\mathbb{H}_{+})\cap H^{1}(\mathbb{H}_{+})$ such that $f=f_{0}+f_{1}I+f_{2}J+f_{3}IJ$. For every $\tilde{I}\in\mathbb{S}$, let $F^{\tilde{I}}(y):=f(\tilde{I}y)$ be the NTBL of $f$ at $\tilde{I}y$ and let $F_{m}^{\tilde{I}}(y):=f_{m}(\tilde{I}y)$ be the NTBL of $f_{m}$ at $\tilde{I}y$, $0\leq m\leq3$. We define the essential Fourier transform as 
        \[\mathscr{F}_{E}(f)(t):=\mathscr{F}_{\tilde{I}}(F^{\tilde{I}})(t)=\mathscr{F}_{\tilde{I}}(F_{0}^{\tilde{I}})(t)+\mathscr{F}_{\tilde{I}}(F_{1}^{\tilde{I}})(t)I+\mathscr{F}_{\tilde{I}}(F_{2}^{\tilde{I}})(t)J+\mathscr{F}_{\tilde{I}}(F_{3}^{\tilde{I}})(t)IJ.\]
        \end{definition}

	\begin{prop}[\cite{alpay2016slice}]\label{prop-RK}
		The kernel $\frac{1}{2\pi}k\left(\mathbf{q}_1,\mathbf{q}_2\right)$ is reproducing, i.e., for any $f\in H^2\left(\mathbb{H}_{+}\right)$ and $I\in\mathbb{S}$,
		\[f\left(\mathbf{q}\right)=\int_{\mathbb{R}}\frac{1}{2\pi}k\left(\mathbf{q},Iy\right)F^{I}\left(y\right)\mathrm{d}y,\ \mathbf{q}\in\mathbb{H}_{+},\]
		where $k\left(\mathbf{q}_1,\mathbf{q}_2\right)=\int_{0}^{+\infty}e^{-\mathbf{q}_1t}e^{-\overline{\mathbf{q}_2}t}\mathrm{d}t$ for $\mathbf{q}_1,\mathbf{q}_2\in\mathbb{H}_{+}$.
		%	where $k\left(p,Iy\right)=\mathscr{F}_I\left(\chi_+\left(\cdot\right)e^{-p\left(\cdot\right)}\right)\left(y\right)$.
	\end{prop}
	For any fixed $I\in\mathbb{S}$, let $f:\Pi_{+,I}\to\mathbb{C}_{I}$. If $f\in H^2\left(\Pi_{+,I}\right)$, then, by the multiplication formula of Fourier transform and Proposition \ref{prop-RK}, we have
	\begin{align*}
		f\left(z\right)&=\int_{\mathbb{R}}\frac{1}{2\pi}k\left(z,Iy\right)F^{I}\left(y\right)\mathrm{d}y\\
		&=\frac{1}{\sqrt{2\pi}}\int_{\mathbb{R}}\mathscr{F}_I\left(\chi_{(-\infty,0]}\left(\cdot\right)e^{z\left(\cdot\right)}\right)\left(t\right)F^{I}\left(t\right)\mathrm{d}t\\
		&=\frac{1}{\sqrt{2\pi}}\int_{-\infty}^{0}e^{zt}\mathscr{F}_{I}\left(F^{I}\right)\left(t\right)\mathrm{d}t,
	\end{align*}
	where $z=x+Iy\in\mathbb{C}_{I}$. It means that $\textbf{supp}\mathscr{F}_{I}\left(F^{I}\right)\subset\left(-\infty,0\right]$.
	\section{Main results}
	Our main results are stated as follows. 
	\subsection{Compact type Paley-Wiener theorem}
	\begin{theorem}\label{PW-QQ2}
		Let $f:\mathbb{H}\to\mathbb{H}$ be a left slice regular function with $f|_\mathbb{R}\in L^2(\mathbb{R})$, i.e.,
		\begin{equation*}
			\int_{-\infty}^{+\infty}|f(x)|^2\mathrm{d}x<\infty,
		\end{equation*}
		and $A>0$ be a positive number. Then the following two conditions are equivalent:
		\begin{enumerate}
			\item[(i)] There exists a constant $C$ such that 
			\begin{equation*}
				|f(\mathbf{q})|\leq Ce^{A|\mathbf{q}|}
			\end{equation*}
			for all $\mathbf{q}\in\mathbb{H}$.
			\item[(ii)] $\operatorname{supp}\mathscr{F}_{I}\left(f|_{\mathbb{R}}\right)\subset\left[-A,A\right]$ for all $I\in\mathbb{S}$.
		\end{enumerate}
		Moreover, if one of the above conditions holds, then we have 
		\begin{equation*}
			f(x+Iy)=\frac{1}{\sqrt{2\pi}}\int_{-A}^{A}e^{I\left(x+Iy\right)t}\mathscr{F}_{I}\left(f|_{\mathbb{R}}\right)(t)\mathrm{d}t.
		\end{equation*}
	\end{theorem}
	In order to prove Theorem \ref{PW-QQ2}, we need the following results.
	\begin{lemma}\label{PW-CICI}
		For any fixed $I\in\mathbb{S}$, let $\Psi:\mathbb{C}_{I}\to\mathbb{C}_{I}$ be a holomorphic function with $\Psi|_{\mathbb{R}}\in L^2(\mathbb{R})$, and $A>0$ be a positive number. Then the following two conditions are equivalent:
		\begin{enumerate}
			\item[(i)] There exists a constant $C$ such that 
			\begin{equation}\label{PW-CICI-fnorm}
				|\Psi(x+Iy)|\leq Ce^{A|x+Iy|},\ for\ all\ x,y\in\mathbb{R}.
			\end{equation}
			\item[(ii)] $\operatorname{supp}\mathscr{F}_I\left(\Psi|_{\mathbb{R}}\right)\subset\left[-A,A\right]$.
		\end{enumerate}
		Moreover, if one of the above conditions holds, then we have 
		\begin{equation}\label{PW-CICI-HFT}
			\Psi(x+Iy)=\frac{1}{\sqrt{2\pi}}\int_{-A}^{A}e^{I\left(x+Iy\right)t}\mathscr{F}_I\left(\Psi|_{\mathbb{R}}\right)(t)\mathrm{d}t.
		\end{equation}
	\end{lemma}
	Based on the argument in \cite[Theorem 19.3]{rudin1987real} (see also \cite{LiangPaley2013}), we can immediately obtain Lemma \ref{PW-CICI}. For completeness, we include a brief proof.
	\begin{proof}
		Assume that (i) holds. Put $\Psi_{\epsilon}\left(x\right)=e^{-\epsilon|x|}\Psi\left(x\right)$, for $\epsilon>0$ and $x\in\mathbb{R}$. By the Cauchy-Schwarz inequality, we have
		\begin{align*}
			\int_{-\infty}^{\infty}\lvert \Psi_{\epsilon}\left(x\right)\rvert\mathrm{d}x&=\int_{-\infty}^{\infty}\lvert e^{-\epsilon|x|}\Psi\left(x\right)\rvert\mathrm{d}x\\
			&\leq\left(\int_{-\infty}^{\infty}e^{-2\epsilon|x|}\mathrm{d}x\right)^{1/2}\left(\int_{-\infty}^{\infty}\lvert \Psi\left(x\right)\rvert^{2}\mathrm{d}x\right)^{1/2}\\
			&=\frac{1}{\epsilon}\lVert \Psi|_{\mathbb{R}}\rVert_{L^2(\mathbb{R})}<\infty.
		\end{align*}
		It means that $\Psi_{\epsilon}\left(x\right)\in L^1(\mathbb{R})$. For fixed $I\in\mathbb{S}$, the Fourier transform of $\Psi_{\epsilon}$ is
		\begin{equation*}
			\mathscr{F}_{I}\left(\Psi_{\epsilon}\right)\left(t\right)=\frac{1}{\sqrt{2\pi}}\int_{-\infty}^{\infty}e^{-Ixt}\Psi_{\epsilon}\left(x\right)\mathrm{d}x.
		\end{equation*}
		Since $\lim_{\epsilon\to0}\lVert \Psi_{\epsilon}-\Psi|_{\mathbb{R}}\rVert_{L^2(\mathbb{R})}=0$, using the Plancherel theorem, we get $\lim_{\epsilon\to0}\lVert\mathscr{F}_{I}\left(\Psi_{\epsilon}\right)-\mathscr{F}_{I}\left(\Psi|_{\mathbb{R}}\right)\rVert_{L^2(\mathbb{R})}=0$. It follows that if 
		\begin{equation}\label{PW-LILI-LFToffe=0}
			\lim_{\epsilon\to0}\frac{1}{\sqrt{2\pi}}\int_{-\infty}^{\infty}e^{-Ixt}\Psi_{\epsilon}\left(x\right)\mathrm{d}x=0
		\end{equation}
		holds for $t\in\mathbb{R}$ and $|t|>A$, then $\mathscr{F}_{I}\left(\Psi|_{\mathbb{R}}\right)$ vanishes outside $\left[-A,A\right]$. It is easy to see that
		\begin{equation*}
			\Psi\left(x\right)=\frac{1}{\sqrt{2\pi}}\int_{-A}^{A}e^{Ixt}\mathscr{F}_{I}\left(\Psi|_{\mathbb{R}}\right)\left(t\right)\mathrm{d}t.
		\end{equation*}
		Since each side of \eqref{PW-CICI-HFT} is an entire function, it follows that \eqref{PW-CICI-HFT} holds for every $z=x+Iy\in \mathbb{C}_{I}$. 
		
		Now, one can show that \eqref{PW-LILI-LFToffe=0} holds for $t\in\mathbb{R}$ and $|t|>A$. For each  $\theta\in\mathbb{R}$, we let $\Gamma_{\theta}$ be the  half line defined by
		\begin{equation*}
			\Gamma_{\theta}(s)=se^{I\theta}\quad(0\leq s<\infty)
		\end{equation*}
		and put
		\begin{equation}\label{w-norm}
			\Pi_{\theta}=\left\{w:\Re(we^{I\theta})>A\right\}.
		\end{equation}
		If $w\in\Pi_{\theta}$, we define
		\begin{equation*}
			\Phi_{\theta}(w)=\int_{\Gamma_{\theta}}e^{-wz}\Psi\left(z\right)\mathrm{d}z=e^{I\theta}\int_{0}^{\infty}e^{-wse^{I\theta}}\Psi\left(se^{I\theta}\right)\mathrm{d}s.
		\end{equation*}
		By (\ref{PW-CICI-fnorm}) and (\ref{w-norm}), we have
		
		\begin{equation*}
			|e^{-wz}\Psi\left(z\right)|\leq Ce^{-[\Re(we^{I\theta})-A]s}
		\end{equation*}
		for $z\in\Gamma_{\theta}$.
		It follows that $\Phi_{\theta}$ is holomorphic.
		
		Note that we have
		\begin{equation*}
			\Phi_0(w)=\int_{0}^{\infty}e^{-ws}\Psi\left(s\right)ds,\quad\Re w>A,
		\end{equation*}
		and
		\begin{equation*}
			\Phi_{\pi}(w)=-\int_{-\infty}^{0}e^{-ws}\Psi\left(s\right)ds,\quad \Re w<-A.
		\end{equation*}
		Using $\Psi|_{\mathbb{R}}\in L^2(\mathbb{R})$, Fubini's theorem and Morera's  theorem, we have $\Phi_0$ and $\Phi_{\pi}$ are holomorphic.
		
		A direct computation gives
		\begin{equation}\label{eps}
			\int_{-\infty}^{+\infty}e^{-Ixt}\Psi_{\epsilon}(x)\mathrm{d}x=\Phi_0(\epsilon+It)-\Phi_{\pi}(-\epsilon+It),\quad t\in\mathbb{R}.
		\end{equation}
		Hence, we have to prove that the right side of (\ref{eps}) tends to $0$ as $\epsilon\to0$, if $t>A$ or $t<-A$. 
		
	Recall that $\Phi_{\frac{\pi}{2}}, \Phi_{-\frac{\pi}{2}}$ are holomorphic on $\Pi_{\frac{\pi}{2}}, \Pi_{-\frac{\pi}{2}}$, respectively. Notice that $\pm \epsilon+It\in\Pi_{\frac{\pi}{2}}$ is equivalent to $t< -A$ and $\pm \epsilon + It \in \Pi_{-\frac{\pi}{2}}$ is equivalent to $t>A$. Next we will show that for $t<-A$,
   \[\Phi_{0}(\pm\epsilon+It)=\Phi_{\pi}(\pm\epsilon+It)=\Phi_{\frac{\pi}{2}}(\pm\epsilon+It),\]
   and for $t>A$,
   \[\Phi_{0}(\pm\epsilon+It)=\Phi_{\pi}(\pm\epsilon+It)=\Phi_{-\frac{\pi}{2}}(\pm\epsilon+It).\]
			
			Assume that $\epsilon,|t|>A$. Let $w_{1}=\sqrt{\epsilon^{2}+t^{2}}e^{-I\frac{\pi}{4}},w_{2}=\sqrt{\epsilon^{2}+t^{2}}e^{-I\frac{3\pi}{4}},w_{3}=\sqrt{\epsilon^{2}+t^{2}}e^{I\frac{3\pi}{4}}$ and $w_{4}=\sqrt{\epsilon^{2}+t^{2}}e^{I\frac{\pi}{4}}.$ Then $w_{1}\in\Pi_{0}\cap\Pi_{\frac{\pi}{2}}$, $w_{2}\in\Pi_{\pi}\cap\Pi_{\frac{\pi}{2}}$, $w_{3}\in\Pi_{\pi}\cap\Pi_{-\frac{\pi}{2}}$ and $w_{4}\in\Pi_{0}\cap\Pi_{-\frac{\pi}{2}}$. 
   
   Let $\Gamma_{r,1}=\{z=re^{I\zeta}:0\leq\zeta\leq\frac{\pi}{2}\},\Gamma_{r,2}=\{z=re^{I\zeta}:\frac{3\pi}{4}\leq\zeta\leq\pi\},\Gamma_{r,3}=\{z=re^{I\zeta}:-\pi\leq\zeta\leq-\frac{\pi}{2}\},$ and $\Gamma_{r,4}=\{z=re^{I\zeta},-\frac{\pi}{2}\leq\zeta\leq0\}.$ Since
			\[|e^{-w_{n}z}\Psi\left(z\right)|\leq Ce^{\left(A-\frac{\sqrt{2}}{2}\sqrt{\epsilon^{2}+t^{2}}\right)r},z\in\Gamma_{r,n}, n=1,2,3,4,\]
            and $A-\frac{\sqrt{2}}{2}\sqrt{\epsilon^{2}+t^{2}}<0$, we obtain that
            \[\lim_{r\to\infty}\int_{\Gamma_{r,n}}e^{-w_{n}z}\Psi\left(z\right)\mathrm{d}z=0, n=1,2,3,4.\]
		  Then, by Cauchy's theorem, we have
            \[\Phi_{0}(w_{1})=\Phi_{\frac{\pi}{2}}(w_{1}),\Phi_{\pi}(w_{2})=\Phi_{\frac{\pi}{2}}(w_{2}),\Phi_{\pi}(w_{3})=\Phi_{-\frac{\pi}{2}}(w_{3}), \text{ and }\Phi_{0}(w_{4})=\Phi_{-\frac{\pi}{2}}(w_{4}).\]
            Hence, by the identity principle, we can get
            \[\Phi_{0}(\pm\epsilon+It)=\Phi_{\pi}(\pm\epsilon+It)=\Phi_{\frac{\pi}{2}}(\pm\epsilon+It),t<-A,\]
            and
            \[\Phi_{0}(\pm\epsilon+It)=\Phi_{\pi}(\pm\epsilon+It)=\Phi_{-\frac{\pi}{2}}(\pm\epsilon+It),t>A.\]
            It follows that the right side of (\ref{eps}) tends to $0$ as $\epsilon\to0$, if $t>A$ or $t<-A$.

		Conversely, we assume that (ii) holds. For any fixed $I\in\mathbb{S}$, let
		\begin{equation*}
			\widetilde{\Psi}(x+Iy)=\frac{1}{\sqrt{2\pi}}\int_{-A}^{A}e^{I\left(x+Iy\right)t}\mathscr{F}_{I}\left(\Psi|_{\mathbb{R}}\right)\left(t\right)\mathrm{d}t.
		\end{equation*}
		It is easy to see that $\widetilde{\Psi}(x+Iy)$ is holomorphic on $\mathbb{C}_{I}$. Using the Cauchy–Schwarz inequality and the Plancherel theorem, we get
		\begin{align*}
			\lvert \widetilde{\Psi}(x+Iy)\rvert&=\Big\lvert\frac{1}{\sqrt{2\pi}}\int_{-A}^{A}e^{I\left(x+Iy\right)t}\mathscr{F}_{I}\left(\Psi|_{\mathbb{R}}\right)\left(t\right)\mathrm{d}t\Big\rvert\\
			&\leq\frac{1}{\sqrt{2\pi}}\int_{-A}^{A}\lvert e^{I\left(x+Iy\right)t}\mathscr{F}_{I}\left(\Psi|_{\mathbb{R}}\right)\left(t\right)\rvert\mathrm{d}t\\
			&\leq\frac{1}{\sqrt{2\pi}}\left(\int_{-A}^{A}\Big\lvert e^{I\left(x+Iy\right)t}\Big\rvert^{2}\mathrm{d}t\right)^{\frac{1}{2}}\left(\int_{-A}^{A}\lvert \mathscr{F}_{I}\left(\Psi|_{\mathbb{R}}\right)\left(t\right)\rvert^{2}\mathrm{d}t\right)^{\frac{1}{2}}\\
			&\leq\frac{\sqrt{2A}}{\sqrt{2\pi}}e^{A|x|}\lVert \Psi|_{\mathbb{R}}\rVert_{L^2(\mathbb{R})}\\
			&\leq Ce^{A|x+Iy|}.
		\end{align*}
		Since $\Psi\left(x\right)=\widetilde{\Psi}\left(x\right)$ and both functions are holomorphic in $\mathbb{C}_{I}$, we conclude that $\Psi\left(x+Iy\right)=\widetilde{\Psi}(x+Iy)$. It follows that $|\Psi\left(z\right)|\leq Ce^{A|z|}$ for all $z=x+Iy\in \mathbb{C}_{I}$.
		
		The proof is completed. 
	\end{proof}
	
    \begin{lemma}\label{PW-NQ}
        Suppose $f\in\mathscr{N}_{l}\left(\mathbb{H}\right)$ with $f|_{\mathbb{R}}\in L^2(\mathbb{R})$, and let $A,C>0$ be two positive constants. Then we have the following results:
        \begin{enumerate}
            \item[(i)] Assume that there exists a quaternionic imaginary unit $I\in\mathbb{S}$, such that $\operatorname{supp}\mathscr{F}_{I}\left(f|_{\mathbb{R}}\right)\subset\left[-A,A\right]$. Then for every $J\in\mathbb{S}$ and $x,y\in\mathbb{R}$, $|f(x+Jy)|\leq Ce^{A|x+Jy|}$. 
            \item[(ii)] Assume that there exists a quaternionic imaginary unit $I\in\mathbb{S}$ such that $|f(x+Iy)|\leq Ce^{A|x+Iy|}$ for all $x,y\in\mathbb{R}$. Then, for every $J\in\mathbb{S}$ and $x,y\in\mathbb{R}$, $\operatorname{supp}\mathscr{F}_{J}\left(f|_{\mathbb{R}}\right)\subset\left[-A,A\right]$ and
		\begin{equation*}
			f\left(x+Jy\right)=\frac{1}{\sqrt{2\pi}}\int_{-A}^{A}e^{J\left(x+Jy\right)t}\mathscr{F}_{J}\left(f|_{\mathbb{R}}\right)\left(t\right)\mathrm{d}t.
		\end{equation*}
        \end{enumerate}
    \end{lemma}

	\begin{proof}
		For a fixed $I\in\mathbb{S}$, since $f\in\mathscr{N}_{l}(\mathbb{H})$, we have $\vert f(x+Jy)\vert=\vert f(x+Iy)\vert$ for any $J\in\mathbb{S}$. Using Lemma \ref{PW-CICI}, we immediately get (i). Moverover, if $\mathscr{F}_I(f|_{\mathbb{R}})$ is supported in $[-A,A]$ then so is $\mathscr{F}_J(f|_{\mathbb{R}})$ for any $J\in\mathbb{S}$, thanks to property (3) in Proposition \ref{prop-FTf|R}. Thus, Lemma \ref{PW-CICI} implies (ii).
	\end{proof}
	
	Now we can prove Theorem \ref{PW-QQ2}.
	\begin{proof}[Proof of Theorem \ref{PW-QQ2}]
		Assume that (i) holds. Since $f\in\mathscr{R}_{l}\left(\mathbb{H}\right)$, by Proposition \ref{prop-NE}, for every $\tilde{I}\in\mathbb{S}$, we have
		\begin{align}\label{PW-QQ2-extf}
			f\left(x+\tilde{I}y\right)=&h_0\left(x+\tilde{I}y\right)+h_1\left(x+\tilde{I}y\right)I+h_2\left(x+\tilde{I}y\right)J+h_3\left(x+\tilde{I}y\right)K,
		\end{align}
		where $h_{m}\in\mathscr{N}_{l}\left(\mathbb{H}\right)$, $0\leq m\leq3$ and $K=IJ$. It means that we can write $h_{m}\left(x+\tilde{I}y\right)=h_{m}^0\left(x+\tilde{I}y\right)+\tilde{I}h_{m}^1\left(x+\tilde{I}y\right)$ where $h_{m}^0$ and $h_{m}^1$ are real-valued functions with $h_{m}^0\left(x-\tilde{I}y\right)=h_{m}^0\left(x+\tilde{I}y\right)$ and $h_{m}^1\left(x-\tilde{I}y\right)=-h_{m}^1\left(x+\tilde{I}y\right)$, $0\leq m\leq3$. It follows that
		\begin{align*}
			&f\left(x+\tilde{I}y\right)\\
			=&\left(h_0^0\left(x+\tilde{I}y\right)+\tilde{I}h_0^1\left(x+\tilde{I}y\right)\right)+\left(h_1^0\left(x+\tilde{I}y\right)+\tilde{I}h_1^1\left(x+\tilde{I}y\right)\right)I\\
			&+\left(h_2^0\left(x+\tilde{I}y\right)+\tilde{I}h_2^1\left(x+\tilde{I}y\right)\right)J+\left(h_3^0\left(x+\tilde{I}y\right)+\tilde{I}h_3^1\left(x+\tilde{I}y\right)\right)K\\
			=&\left(h_0^0\left(x+\tilde{I}y\right)+h_1^0\left(x+\tilde{I}y\right)I+h_2^0\left(x+\tilde{I}y\right)J+h_3^0\left(x+\tilde{I}y\right)K\right)\\
			&+\tilde{I}\left(h_0^1\left(x+\tilde{I}y\right)+h_1^1\left(x+\tilde{I}y\right)I+h_2^1\left(x+\tilde{I}y\right)J+h_3^1\left(x+\tilde{I}y\right)K\right)\\
			=&\alpha(x,y)+\tilde{I}\beta(x,y).
		\end{align*}
		Since $\Big|f\left(x+\tilde{I}y\right)\Big|\leq Ce^{A|x+\tilde{I}y|}$ holds for every $\tilde{I}\in\mathbb{S}$, by Proposition \ref{prop-f=a+Ib} and Proposition \ref{prop-RF}, for $I'\in\mathbb{S}$, we get
		\begin{equation*}
			|\alpha(x,y)|\leq\frac{1}{2}\left[\Big\lvert f\left(x+I'y\right)\Big\rvert+\Big\lvert f\left(x-I'y\right)\Big\rvert\right]\leq Ce^{A|x+I'y|}=Ce^{A|x+\tilde{I}y|}
		\end{equation*}
		and
		\begin{equation*}
			|\beta(x,y)|\leq\frac{1}{2}\left[\Big\lvert f\left(x+I'y\right)\Big\rvert+\Big\lvert f\left(x-I'y\right)\Big\rvert\right]\leq Ce^{A|x+I'y|}=Ce^{A|x+\tilde{I}y|}.
		\end{equation*}
		It is obvious that $\Big|h_{m}^0\left(x+\tilde{I}y\right)\Big|\leq|\alpha(x,y)|\leq Ce^{A|x+\tilde{I}y|}$ and $\Big|h_{m}^1\left(x+\tilde{I}y\right)\Big|\leq|\beta(x,y)|\leq Ce^{A|x+\tilde{I}y|}$, $0\leq m\leq3$. Hence, we have $$\Big|h_{m}\left(x+\tilde{I}y\right)\Big|\leq|\alpha(x,y)|+|\beta(x,y)|\leq C'e^{A|x+\tilde{I}y|},0\leq m\leq3.$$
		Since $f|_{\mathbb{R}}\in L^2(\mathbb{R})$, we have $h_{m}|_{\mathbb{R}}\in L^2(\mathbb{R})$, $0\leq m\leq3$.
		By Lemma \ref{PW-NQ}, one has $\operatorname{supp}\mathscr{F}_{\tilde{I}}\left(h_{m}|_{\mathbb{R}}\right)\subset\left[-A,A\right]$ and
		\begin{equation}\label{PW-QQ2-hl}
			h_{m}\left(x+\tilde{I}y\right)=\frac{1}{\sqrt{2\pi}}\int_{-A}^{A}e^{\tilde{I}\left(x+\tilde{I}y\right)t}\mathscr{F}_{\tilde{I}}(h_{m}|_{\mathbb{R}})(t)\mathrm{d}t,\ m=0,1,2,3.
		\end{equation}
		Combining (\ref{PW-QQ2-extf}) and (\ref{PW-QQ2-hl}), we have
		\begin{align*}
			&f\left(x+\tilde{I}y\right)\\
			=&h_0\left(x+\tilde{I}y\right)+h_1\left(x+\tilde{I}y\right)I+h_2\left(x+\tilde{I}y\right)J+h_3\left(x+\tilde{I}y\right)K\notag\\
			=&\frac{1}{\sqrt{2\pi}}\int_{-A}^{A}e^{\tilde{I}\left(x+\tilde{I}y\right)t}\mathscr{F}_{\tilde{I}}(h_{0}|_{\mathbb{R}})(t)\mathrm{d}t+\frac{1}{\sqrt{2\pi}}\int_{-A}^{A}e^{\tilde{I}\left(x+\tilde{I}y\right)t}\mathscr{F}_{\tilde{I}}(h_{1}|_{\mathbb{R}})(t)\mathrm{d}t I\notag\\
			&+\frac{1}{\sqrt{2\pi}}\int_{-A}^{A}e^{\tilde{I}\left(x+\tilde{I}y\right)t}\mathscr{F}_{\tilde{I}}(h_{2}|_{\mathbb{R}})(t)\mathrm{d}t J+\frac{1}{\sqrt{2\pi}}\int_{-A}^{A}e^{\tilde{I}\left(x+\tilde{I}y\right)t}\mathscr{F}_{\tilde{I}}(h_{3}|_{\mathbb{R}})(t)\mathrm{d}t K\notag\\
			=&\frac{1}{\sqrt{2\pi}}\int_{-A}^{A}e^{\tilde{I}\left(x+\tilde{I}y\right)t}\mathscr{F}_{\tilde{I}}(h_{0}|_{\mathbb{R}}+h_{1}|_{\mathbb{R}}I+h_{2}|_{\mathbb{R}}J+h_{3}|_{\mathbb{R}}K)(t)\mathrm{d}t\notag\\
			=&\frac{1}{\sqrt{2\pi}}\int_{-A}^{A}e^{\tilde{I}\left(x+\tilde{I}y\right)t}\mathscr{F}_{\tilde{I}}(f|_{\mathbb{R}})(t)\mathrm{d}t
		\end{align*}
		and $\operatorname{supp}\mathscr{F}_{\tilde{I}}\left(f|_{\mathbb{R}}\right)\subset\bigcup_{m=0}^{3}\operatorname{supp}\mathscr{F}_{\tilde{I}}(h_{m}|_{\mathbb{R}})\subset\left[-A,A\right]$.
		
		Assume that (ii) holds. For any fixed $\tilde{I}\in\mathbb{S}$, let
		\begin{equation*}
			\Psi\left(x+\tilde{I}y\right)=\frac{1}{\sqrt{2\pi}}\int_{-A}^{A}e^{\tilde{I}\left(x+\tilde{I}y\right)t}\mathscr{F}_{\tilde{I}}\left(f|_{\mathbb{R}}\right)\left(t\right)\mathrm{d}t.
		\end{equation*}
		It is easy to see that $\overline{\partial_{\widetilde{I}}}\Psi\left(x+\tilde{I}y\right)=0$. Using the H$\rm\ddot{o}$lder inequality and the Plancherel theorem, we have
		\begin{align*}
			\lvert \Psi\left(x+\tilde{I}y\right)\rvert&\leq\frac{1}{\sqrt{2\pi}}\int_{-A}^{A}\lvert e^{\tilde{I}\left(x+\tilde{I}y\right)t}\mathscr{F}_{\tilde{I}}\left(f|_{\mathbb{R}}\right)\left(t\right)\rvert\mathrm{d}t\\
			&\leq\frac{1}{\sqrt{2\pi}}\left(\int_{-A}^{A}\lvert e^{\tilde{I}\left(x+\tilde{I}y\right)t}\rvert^{2}\mathrm{d}t\right)^{\frac{1}{2}}\left(\int_{-A}^{A}\lvert \mathscr{F}_{\tilde{I}}\left(f|_{\mathbb{R}}\right)\left(t\right)\rvert^{2}\mathrm{d}t\right)^{\frac{1}{2}}\\
			&\leq\frac{\sqrt{2A}}{\sqrt{2\pi}}e^{A|y|}\lVert f|_{\mathbb{R}}\rVert_{L^2(\mathbb{R})}=Ce^{A|x+\tilde{I}y|}.
		\end{align*}
		Since $\Psi(x)=f(x)$ and both functions are holomorphic in $\mathbb{C}_{\tilde{I}}$, we conclude that $f\left(x+\tilde{I}y\right)=\Psi\left(x+\tilde{I}y\right)$. By Proposition \ref{prop-fIinBR-finB2R}, for every $J\in\mathbb{S}$, we have $$\lvert f\left(x+Jy\right)\rvert\leq C'e^{A|x+\tilde{I}y|}=C'e^{A|x+Jy|}.$$
		
		The proof is completed.
	\end{proof}
	\subsection{Non-compact type Paley-Wiener theorem} 
	Non-compact type Paley-Wiener theorem is stated as follows.
	\begin{theorem}\label{PW-H+H-HPSFT}
		Let $f\in H^{p}\left(\mathbb{H}_{+}\right)$, $1\leq p\leq2$. Then $\operatorname{supp}\mathscr{F}_{E}\left(f\right)\subset\left(-\infty,0\right]$. Moreover, for every $I\in\mathbb{S}$ and $x>0,y\in\mathbb{R}$, 
		\begin{equation*}
			f\left(x+Iy\right)=\frac{1}{\sqrt{2\pi}}\int_{-\infty}^{0}e^{\left(x+Iy\right)t}\mathscr{F}_{E}\left(f\right)\left(t\right)\mathrm{d}t.
		\end{equation*}
	\end{theorem}
	\begin{theorem}\label{PW-H+H-SFTHP}
		Let $F:\mathbb{R}\to\mathbb{H}$ and  $F\in L^{p}(\mathbb{R})$ for $1\leq p\leq2$. Assume that there exists a quaternionic imaginary unit $I\in\mathbb{S}$ such that $\operatorname{supp}\mathscr{F}_{I}\left(F\right)\subset\left(-\infty,0\right]$. Then there exists a function $f\in H^p(\mathbb{H}_+)$, whose non tangential boundary boundary limit at $Iy$ is $F(y)$, whence $\mathscr{F}_{E}\left(f\right)=\mathscr{F}_{I}\left(F\right)$ and
        \begin{equation*}
			f\left(x+Jy\right):=\frac{1}{\sqrt{2\pi}}\int_{-\infty}^{0}e^{\left(x+Jy\right)t}\mathscr{F}_{I}\left(F\right)\left(t\right)\mathrm{d}t,
		\end{equation*}
        for all $x>0,y\in\mathbb{R}$, $J\in\mathbb{S}$.
	\end{theorem}

	In order to prove Theorem \ref{PW-H+H-HPSFT} and Theorem \ref{PW-H+H-SFTHP}, we need the following results.
	\begin{lemma}\label{1<cHp}
		For any fixed $I\in\mathbb{S}$, let $\Psi:\Pi_{+,I}\to\mathbb{C}_{I}$. Suppose $\Psi\in H^p\left(\Pi_{+,I}\right),0<p\leq\infty$, and $x_{0}>0$. Let  $\Pi_{+,I,x_0}=\left\{x_0+z:z\in\Pi_{+,I}\right\}\subset\Pi_{+,I}$, then there exists a constant $C=C\left(x_0,p\right)$, depending on $\left(x_0,p\right)$ but not on $\Psi$, such that
		\[\sup_{z\in\Pi_{+,I,x_0}}|\Psi\left(z\right)|\leq C\lVert \Psi\rVert_{H^p}.\]
	\end{lemma}
	\begin{proof}
		The proof is similar to that given in \cite[Proposition 1.3]{bekolle2003lecture}.
	\end{proof}
	\begin{lemma}\label{FTLp}
		For any fixed $I\in\mathbb{S}$, let $\Psi,\Phi:\mathbb{C}_{I}\to\mathbb{C}_{I}$. If $\Psi|_{\mathbb{R}}\in L^p(\mathbb{R})$, $\Phi|_{\mathbb{R}}\in L^p(\mathbb{R})\cap L^1(\mathbb{R})$, $1\leq p\leq2$, then
		\begin{equation*}
			\int_{\mathbb{R}}\Psi\left(t\right)\mathscr{F}_I\left(\Phi|_{\mathbb{R}}\right)\left(t\right)\mathrm{d}t=\int_{\mathbb{R}}\mathscr{F}_I\left(\Psi|_{\mathbb{R}}\right)\left(t\right)\Phi\left(t\right)\mathrm{d}t.
		\end{equation*}
	\end{lemma}
	\begin{proof}
		The proof is similar to that given in \cite{li2018fourier}. 
	\end{proof}
 
	\begin{lemma}\label{PW-PICI-Lp}
		For any fixed $I\in\mathbb{S}$, let $F:\mathbb{R}\to\mathbb{C}_{I}$. Then $F\left(y\right)$ is the boundary limit function of $\Psi(x+Iy)\in H^p\left(\Pi_{+,I}\right)$, $1\leq p\leq2$ if and only if $F\in L^p(\mathbb{R})$, $\operatorname{supp}\mathscr{F}_I\left(F\right)\subset\left(-\infty,0\right]$ and for all $z=x+Iy\in\Pi_{+,I}$,
    \begin{equation}\label{pw-PICI-Lp-psi}
        \Psi\left(z\right)=\frac{1}{\sqrt{2\pi}}\int_{-\infty}^{0}e^{zt}\mathscr{F}_I\left(F\right)\left(t\right)\mathrm{d}t=\int_{\mathbb{R}}K\left(z-It\right)F(t)\mathrm{d}t=\int_{\mathbb{R}}P\left(x,y-t\right)F\left(t\right)\mathrm{d}t
    \end{equation}
    where 
		\begin{equation}\label{pw-PICI-Lp-CKPK}
			K\left(z\right)=\frac{1}{2\pi}\int_{-\infty}^{0}e^{z\xi}\mathrm{d}\xi,\quad P\left(x,y\right)=\frac{|K\left(z\right)|^2}{K\left(2x\right)}.
		\end{equation}
	\end{lemma}
	\begin{proof}
		The proof is similar to that given in \cite{li2018fourier}. Let us first prove that, if $\Psi\left(x+Iy\right)\in H^p\left(\Pi_{+,I}\right)$, then $\operatorname{supp}\mathscr{F}_I\left(F\right)\subset\left(-\infty,0\right]$.
		
		For any fixed $x_0>0$, we write
		\begin{equation*}
			\Psi_{x_0}\left(z\right)=\Psi\left(x_0+z\right).
		\end{equation*}
		By Lemma \ref{1<cHp}, there exists a positive constant C=C$\left(x_0,p\right)$ such that
		\begin{equation*}
			\sup_{z\in\Pi_{+,I}}|\Psi_{x_0}\left(z\right)|\leq C\lVert \Psi\rVert_{H^p\left(\Pi_{+,I}\right)}<\infty.
		\end{equation*}
		From the last inequality, we have, for any $x>0$,
		\begin{align*}
			\int_{\mathbb{R}}\lvert \Psi_{x_0}\left(x+Iy\right)\rvert^2\mathrm{d}y&=\int_{\mathbb{R}}\lvert \Psi_{x_0}\left(x+Iy\right)\rvert^p\lvert \Psi_{x_0}\left(x+Iy\right)\rvert^{2-p}\mathrm{d}y\notag\\
			&\leq\int_{\mathbb{R}}\lvert \Psi_{x_0}\left(x+Iy\right)\rvert^p\left(C\lVert \Psi\rVert_{H^p\left(\Pi_{+,I}\right)}\right)^{2-p}\mathrm{d}y\notag\\
			&\leq\left(C\lVert \Psi\rVert_{H^p\left(\Pi_{+,I}\right)}\right)^{2-p}\lVert \Psi_{x_0}\rVert_{p}^{p}<\infty,
		\end{align*}
		so that
		\begin{equation*}
			\Psi_{x_0}\left(z\right)\in H^2\left(\Pi_{+,I}\right).
		\end{equation*}
		Proposition \ref{prop-RK} implies that
		\begin{equation}\label{PW-PICI-Lp-SFTfx0}
			\operatorname{supp}\mathscr{F}_I\left(F_{x_{0}}\right)\subset\left(-\infty,0\right],
		\end{equation}
  where $F_{x_{0}}(y)=\Psi_{x_{0}}(Iy)$ is the NTBL of $\Psi_{x_{0}}$.
		When $p=1$, we have
		\begin{equation*}
			\lVert\mathscr{F}_I\left(F\right)\rVert_{L^{\infty}(\mathbb{R})}\leq\lVert F\rVert_{L^{1}(\mathbb{R})}.
		\end{equation*}
		When $1<p\leq2$, by Hausdorff-Young's inequality, for $\frac{1}{p}+\frac{1}{p^\prime}=1$, we have
		\begin{equation*}
			\lVert\mathscr{F}_I\left(F\right)\rVert_{L^{p^\prime}(\mathbb{R})}\leq\lVert F\rVert_{L^{p}(\mathbb{R})}.
		\end{equation*}
		Hence, for $1\leq p\leq2$, by the above two inequalities, we have
		\begin{equation}\label{PW-PICI-Lp-FTfx0-FTf}
			\lVert\mathscr{F}_I\left(F_{x_{0}}\right)-\mathscr{F}_I\left(F\right)\rVert_{L^{p^\prime}(\mathbb{R})}\leq\lVert F_{x_{0}}-F\rVert_{L^{p}(\mathbb{R})}.
		\end{equation}
		Since $$\lim_{x_0\to0}\lVert F_{x_{0}}-F\rVert_{L^{p}(\mathbb{R})}=0,$$ we have $$\lim_{x_0\to0}\lVert\mathscr{F}_I\left(F_{x_{0}}\right)-\mathscr{F}_I\left(F\right)\rVert_{L^{p^\prime}(\mathbb{R})}=0.$$
		Due to (\ref{PW-PICI-Lp-SFTfx0}) and (\ref{PW-PICI-Lp-FTfx0-FTf}), for any $\varphi\in\mathscr{S}\left(\mathbb{R}\right)$, where $\mathscr{S}\left(\mathbb{R}\right)$ is the space of rapidly decreasing functions on $\mathbb{R}$, satisfying $\operatorname{supp}\varphi\subset[0,+\infty)$, we have
		\begin{align*}
			&\Big|\int_{\mathbb{R}}\mathscr{F}_I\left(F\right)\left(t\right)\varphi\left(t\right)\mathrm{d}t\Big|\\
			=&\Big|\int_{\mathbb{R}}\left(\mathscr{F}_I\left(F\right)\left(t\right)-\mathscr{F}_I\left(F_{x_{0}}\right)\left(t\right)\right)\varphi\left(t\right)\mathrm{d}t+\int_{\mathbb{R}}\mathscr{F}_I\left(F_{x_{0}}\right)\left(t\right)\varphi\left(t\right)\mathrm{d}t\Big|\notag\\
			=&\Big|\int_{\mathbb{R}}\left(\mathscr{F}_I\left(F\right)\left(t\right)-\mathscr{F}_I\left(F_{x_{0}}\right)\left(t\right)\right)\varphi\left(t\right)\mathrm{d}t\Big|\notag\\
			\leq&\Big\lVert\mathscr{F}_I\left(F\right)-\mathscr{F}_I\left(F_{x_{0}}\right)\Big\rVert_{L^{p^\prime}(\mathbb{R})}\lVert\varphi\rVert_{L^p(\mathbb{R})}.
		\end{align*}
		This implies that $\operatorname{supp}\mathscr{F}_I\left(F\right)\subset\left(-\infty,0\right]$.
		
		Let us now prove that if $F\in L^p(\mathbb{R})$ and $\operatorname{supp}\mathscr{F}_I\left(F\right)\subset\left(-\infty,0\right]$, then Formula \eqref{pw-PICI-Lp-psi} holds, and, as a consequence, $\Psi\left(x+Iy\right)\in H^p\left(\Pi_{+,I}\right)$.
		
		Let $F\in L^p(\mathbb{R})$, $1\leq p\leq2$, and $\operatorname{supp}\mathscr{F}_I\left(F\right)\subset\left(-\infty,0\right]$. First, we can easily show that for any fixed $x>0$ and for $z=x+Iy$,
		\begin{equation*}
			|e^{zt}\mathscr{F}_I\left(F\right)\left(t\right)|=|\mathscr{F}_I\left(F\right)\left(t\right)|e^{xt}\in L^1\left(-\infty,0\right).
		\end{equation*}
		The Lebesgue dominated convergence theorem and the H$\rm\ddot{o}$lder inequality can be used to conclude that the function
		\begin{equation}\label{PW-PICI-Lp-gz}
			\Psi\left(z\right)=\frac{1}{\sqrt{2\pi}}\int_{-\infty}^{0}e^{zt}\mathscr{F}_I\left(F\right)\left(t\right)\mathrm{d}t
		\end{equation}
		is continuous in $\Pi_{+,I}$. The Fubini theorem and the Cauchy integral theorem show that $\int_{\gamma}\Psi\left(z\right)\mathrm{d}z=0$ for every closed path $\gamma$ in $\Pi_{+,I}$. Using the Morera theorem, we have that $\Psi$ is holomorphic in $\Pi_{+,I}$.
		
		Next, we show $\Psi(x+Iy)\in H^{p}\left(\Pi_{+,I}\right)$. For $z,w\in\Pi_{+,I}$, let
		\begin{equation}\label{PW-PICI-Lp-I(z,w)}
			I\left(z,w\right)=\int_{\mathbb{R}}K\left(z-It\right)F\left(t\right)K\left(w+It\right)\mathrm{d}t.
		\end{equation}
		By (\ref{pw-PICI-Lp-CKPK}), we have $$K\left(z-It\right)=\frac{1}{\sqrt{2\pi}}\mathscr{F}_I\left(\chi_-\left(\cdot\right)e^{z\left(\cdot\right)}\right)\left(t\right)$$ and $$K\left(w+It\right)=\frac{1}{\sqrt{2\pi}}\mathscr{F}_I\left(\chi_{+}\left(\cdot\right)e^{-w\left(\cdot\right)}\right)\left(t\right),$$ where
		$\chi_{\pm}\left(t\right)$ are characteristic functions on $[0,\infty)$ and $(-\infty,0]$ respectively. Note that $\operatorname{supp}\mathscr{F}_{I}(F)\subset(-\infty,0]$, which means that $\mathscr{F}_{I}(F)(t)=\chi_{-}(t)\mathscr{F}_{I}(F)(t)$. Since 
		\begin{equation*}
			\mathscr{F}_I\left(\chi_-\left(\cdot\right)e^{z\left(\cdot\right)}\right)\left(t\right)\mathscr{F}_I\left(\chi_{+}\left(\cdot\right)e^{-w\left(\cdot\right)}\right)\left(t\right)=\mathscr{F}_I\left(\chi_-\left(\cdot\right)e^{z\left(\cdot\right)}*\chi_{+}\left(\cdot\right)e^{-w\left(\cdot\right)}\right)\left(t\right),
		\end{equation*}
		and Lemma \ref{FTLp} we have
		\begin{align*}
			I\left(z,w\right)=&\frac{1}{2\pi}\int_{\mathbb{R}}\chi_{-}\left(t\right)\mathscr{F}_I\left(F\right)\left(t\right)\int_{\mathbb{R}}\chi_{-}\left(\xi\right)e^{z\xi}\chi_{+}\left(t-\xi\right)e^{-w\left(t-\xi\right)}\mathrm{d}\xi\mathrm{d}t\\
			=&\frac{1}{2\pi}\int_{\mathbb{R}}\chi_{-}\left(t\right)\mathscr{F}_I\left(F\right)\left(t\right)\int_{\mathbb{R}}\chi_{-}\left(\xi\right)e^{z\xi}\chi_{-}\left(\xi-t\right)e^{w\left(\xi-t\right)}\mathrm{d}\xi\mathrm{d}t.
		\end{align*}
		Let $s=\xi-t$. Using $\chi_{-}\left(t\right)\chi_{-}\left(s+t\right)\chi_{-}\left(s\right)=\chi_{-}\left(t\right)\chi_{-}\left(s\right)$, we have 
		\begin{align*}
			I\left(z,w\right)=&\frac{1}{2\pi}\int_{\mathbb{R}}\chi_{-}\left(t\right)\mathscr{F}_I\left(F\right)\left(t\right)\int_{\mathbb{R}}\chi_{-}\left(s\right)e^{z\left(s+t\right)}e^{w s}\mathrm{d}s\mathrm{d}t\\
			=&\frac{1}{2\pi}\int_{\mathbb{R}}\chi_{-}\left(t\right)e^{zt}\mathscr{F}_I\left(F\right)\left(t\right)\int_{\mathbb{R}}\chi_{-}\left(s\right)e^{\left(z+w\right)s}\mathrm{d}s\mathrm{d}t.
		\end{align*}
        As a consequence, $I(z,w)=\Psi(z)K(z+w)$. 
		For $z=x+Iy\in\Pi_{+,I}$, we have $\overline{z}=x-Iy\in\Pi_{+,I}$, and
		\begin{equation}\label{PW-PICI-Lp-I=gK}
			I\left(z,\overline{z}\right)=\Psi\left(z\right)K\left(z+\overline{z}\right)=\Psi\left(z\right)K\left(2x\right).
		\end{equation}
		Hence, by the formulas (\ref{PW-PICI-Lp-I(z,w)}) and (\ref{PW-PICI-Lp-I=gK}), we have
		\begin{equation}\label{PW-PICI-Lp-g=Pf}
			\Psi\left(z\right)=\int_{\mathbb{R}}F\left(t\right)\frac{K\left(z-It\right)K\left(\overline{z}+It\right)}{K\left(2x\right)}\mathrm{d}t=\int_{\mathbb{R}}P\left(x,y-t\right)F\left(t\right)\mathrm{d}t,
		\end{equation}
		where $P\left(x,y\right)$ is the Poisson kernel associated with the right half-plane $\Pi_{+,I}$. Due to (\ref{PW-PICI-Lp-g=Pf}), by $\lVert \Psi(x+I\cdot)\rVert_{L^p(\mathbb{R})}=\lVert P*F\rVert_{L^p(\mathbb{R})}\leq\lVert P\rVert_{L^1(\mathbb{R})}\lVert F\rVert_{L^p(\mathbb{R})}$, we have
		\begin{equation*}
			\sup_{x>0}\int_{\mathbb{R}}|\Psi(x+Iy)|^p\mathrm{d}y<\infty.
		\end{equation*}
		It follows that $\Psi(x+Iy)\in H^{p}\left(\Pi_{+,I}\right)$.
		
		Hence by  (\ref{PW-PICI-Lp-g=Pf}), we have $\lim_{x\to0}\Psi(x+Iy)=F(y)$, a.e. $y\in\mathbb{R}$. The function $\Psi$ belongs to $H^{p}(\Pi_{+,I})$ and $F(y)$ is its NTBL. 
		
		Moreover, by Lemma \ref{FTLp} and (\ref{PW-PICI-Lp-gz}), we have
		\begin{align*}
			\Psi\left(z\right)&=\frac{1}{\sqrt{2\pi}}\int_{-\infty}^{0}e^{zt}\mathscr{F}_I\left(F\right)\left(t\right)\mathrm{d}t\notag\\
			&=\frac{1}{\sqrt{2\pi}}\int_{\mathbb{R}}\mathscr{F}_I\left(\chi_{-}\left(\cdot\right)e^{z\left(\cdot\right)}\right)\left(t\right)F\left(t\right)\mathrm{d}t\notag\\
			&=\int_{\mathbb{R}}K\left(z-It\right)F\left(t\right)\mathrm{d}t.
		\end{align*}
		The proof is completed.
	\end{proof}

	\begin{lemma}\label{PW-NH+-HPF}
		Let $f\in\mathscr{N}_{l}\left(\mathbb{H}_{+}\right)\cap H^{p}\left(\mathbb{H}_{+}\right)$, $1\leq p\leq2$. Then $\operatorname{supp}\mathscr{F}_{E}(f)\subset\left(-\infty,0\right]$ and 
		\[f\left(x+Jy\right)=\frac{1}{\sqrt{2\pi}}\int_{-\infty}^{0}e^{\left(x+Jy\right)t}\mathscr{F}_{E}\left(f\right)\left(t\right)\mathrm{d}t\]
  for all $x>0,y\in\mathbb{R}$ and $J\in\mathbb{S}$.
	\end{lemma}
	\begin{proof}
		For any fixed $I\in\mathbb{S}$, by $f\in\mathscr{N}_{l}\left(\mathbb{H}_{+}\right)$, we have $f\left(\Pi_{+,I}\right)\subset\mathbb{C}_{I}$ and $f\left(x-Iy\right)=\overline{f\left(x+Iy\right)}$. Since $f\in H^{p}(\mathbb{H}_{+})$, $1\leq p\leq 2$, we can get $f\left(x+Iy\right)\in H^{p}\left(\Pi_{+,I}\right)$. By Lemma \ref{PW-PICI-Lp}, we have $\operatorname{supp}\mathscr{F}_{I}\left(F^{I}\right)\subset\left(-\infty,0\right]$ and
		\begin{equation*}
			f\left(x+Iy\right)=\frac{1}{\sqrt{2\pi}}\int_{-\infty}^{0}e^{\left(x+Iy\right)t}\mathscr{F}_{I}\left(F^{I}\right)\left(t\right)\mathrm{d}t
		\end{equation*}
        for all $x>0$ and $y\in\mathbb{R}$. It is known that if $F^{I}\in L^{p}(\mathbb{R})$ for $1\le p\le2$ then $F^{I}$ can be written as $F^{I}=F^{I}_{1}+F^{I}_{2}$ where $F^{I}_{1}\in L^{1}(\mathbb{R})$ and $F^{I}_{2}\in L^{2}(\mathbb{R})$. Thus, by Proposition \ref{prop-FTf_I^+} and Plancherel's theorem, there exists a real-valued function $\mathscr{F}_{E}(f)$ which is independent of $I$ such that  $\mathscr{F}_{E} (f)=\mathscr{F}_{I}\left(F^{I}\right)$. Note that for $t$ fixed, $e^{\mathbf{q}t}$ is slice preserving on $\mathbb{H}_{+}$ and decomposes at $x+Iy$ as $e^{xt}\cos{(yt)}+Ie^{xt}\sin{(yt)}$. Then we have
        \begin{align*}
            f\left(x+Iy\right)=&\frac{1}{\sqrt{2\pi}}\int_{-\infty}^{0}e^{\left(x+Iy\right)t}\mathscr{F}_{E}\left(f\right)\left(t\right)\mathrm{d}t\\
            =&\frac{1}{\sqrt{2\pi}}\int_{-\infty}^{0}\left[e^{xt}\cos(yt)+Ie^{xt}\sin(yt)\right]\mathscr{F}_{E}\left(f\right)\left(t\right)\mathrm{d}t\\
            =&\frac{1}{\sqrt{2\pi}}\int_{-\infty}^{0}e^{xt}\cos(yt)\mathscr{F}_{E}\left(f\right)\left(t\right)\mathrm{d}t+I\frac{1}{\sqrt{2\pi}}\int_{-\infty}^{0}e^{xt}\sin(yt)\mathscr{F}_{E}\left(f\right)\left(t\right)\mathrm{d}t\\
            =&\alpha(x,y)+I\beta(x,y),
        \end{align*}
        where $\alpha$ and $\beta$ are two real-valued functions. Hence, for every $J\in\mathbb{S}$, we have
        \begin{align*}
			f\left(x+Jy\right)=&\alpha(x,y)+J\beta(x,y)\\
            =&\frac{1}{\sqrt{2\pi}}\int_{-\infty}^{0}e^{xt}\cos\left(yt\right)\mathscr{F}_{E}\left(f\right)(t)\mathrm{d}t+\frac{J}{\sqrt{2\pi}}\int_{-\infty}^{0}e^{xt}\sin\left(yt\right)\mathscr{F}_{E}\left(f\right)(t)\mathrm{d}t\\
            =&\frac{1}{\sqrt{2\pi}}\int_{-\infty}^{0}\left[e^{xt}\cos\left(yt\right)+Je^{xt}\sin\left(yt\right)\right]\mathscr{F}_{E}\left(f\right)(t)\mathrm{d}t\\
            =&\frac{1}{\sqrt{2\pi}}\int_{-\infty}^{0}e^{\left(x+Jy\right)t}\mathscr{F}_{E}\left(f\right)(t)\mathrm{d}t
		\end{align*}
		and $\operatorname{supp}\mathscr{F}_{E}\left(f\right)=\operatorname{supp}\mathscr{F}_{I}\left(F^{I}\right)\subset\left(-\infty,0\right]$.
	\end{proof}
    \begin{lemma}
        Let $\mathscr{A},\mathscr{B}:\mathbb{R}\to\mathbb{R}$ with $\mathscr{A}(-y)=\mathscr{A}(y)$ and $\mathscr{B}(-y)=-\mathscr{B}(y)$, and $\mathscr{A},\mathscr{B}\in L^{p}(\mathbb{R}), 1\leq p\leq2$. Assume that there exists a quaternionic imaginary unit $I\in\mathbb{S}$ such that $\operatorname{supp}\mathscr{F}_{I}(\mathscr{A}),\operatorname{supp}\mathscr{F}_{I}(\mathscr{B})\subset(-\infty,0]$. Then for all $x>0,y\in\mathbb{R}$ and $J\in\mathbb{S}$,
        \[f(x+Jy)=\frac{1}{\sqrt{2\pi}}\int_{-\infty}^{0}e^{(x+Jy)t}\mathscr{F}_{J}\left(\mathscr{A}+J\mathscr{B}\right)(t)\mathrm{d}t\]
		is in $H^{p}\left(\mathbb{H}_{+}\right)\cap\mathscr{N}_{l}\left(\mathbb{H}_{+}\right)$. Moreover, let $F^{J}(y)$ be the NTBL of $f$ at $Jy$, then $$\mathscr{A}(y)=\frac{1}{2}\left[F^{J}(y)+F^{J}(-y)\right]=\alpha(0,y)$$ and $$\mathscr{B}(y)=\frac{J}{2}\left[F^{J}(-y)-F^{J}(y)\right]=\beta(0,y)$$ holds a.e. $y\in\mathbb{R}$, and $\mathscr{F}_{E}(f)=\mathscr{F}_{J}\left(\mathscr{A}+J\mathscr{B}\right)$ is independent of $J$.
    \end{lemma}
	
	\begin{proof}
		For a fixed $I\in\mathbb{S}$, since $\mathscr{A},\mathscr{B}\in L^{p}(\mathbb{R})$ for $1\leq p\leq2$ and $\operatorname{supp}\mathscr{F}_{I}\left(\mathscr{A}\right),\operatorname{supp}\mathscr{F}_{I}\left(\mathscr{B}\right)\subset\left(-\infty,0\right]$, we have $\mathscr{A}+I\mathscr{B}\in L^{p}(\mathbb{R})$ and $\operatorname{supp}\mathscr{F}_{I}\left(\mathscr{A}+I\mathscr{B}\right)\subset\left(-\infty,0\right]$. Thus, by Lemma \ref{PW-PICI-Lp}, we have that 
  \begin{equation}\label{PW-LpHp-Pf}
			\Psi\left(x+Iy\right)=\int_{\mathbb{R}}P\left(x,y-t\right)\left[\mathscr{A}\left(t\right)+I\mathscr{B}\left(t\right)\right]\mathrm{d}t
		\end{equation}
  is in $H^{p}(\Pi_{+,I})$ where $\mathscr{A}\left(y\right)+I\mathscr{B}\left(y\right)=\Psi(Iy)$ denotes the non-tangential value of $\Psi$ at $Iy$.
		Taking the conjugate of both sides of \eqref{PW-LpHp-Pf}, we have
		\begin{align*}
			\overline{\Psi\left(x+Iy\right)}&=\overline{\int_{\mathbb{R}}P\left(x,y-t\right)\left[\mathscr{A}\left(t\right)+I\mathscr{B}\left(t\right)\right]\mathrm{d}t}\notag\\
			&=\int_{\mathbb{R}}P\left(x,y-t\right)\overline{\left[\mathscr{A}\left(t\right)+I\mathscr{B}\left(t\right)\right]}\mathrm{d}t.
		\end{align*}
		Since $\mathscr{A}\left(-y\right)=\mathscr{A}\left(y\right)$ and $\mathscr{B}(-y)=-\mathscr{B}(y)$, we have
		\begin{align*}
			\overline{\Psi\left(x+Iy\right)}&=\int_{\mathbb{R}}P\left(x,y-t\right)\left[\mathscr{A}\left(-t\right)+I\mathscr{B}\left(-t\right)\right]\mathrm{d}t\notag\\
			&=\int_{\mathbb{R}}P\left(x,y+t\right)\left[\mathscr{A}\left(t\right)+I\mathscr{B}\left(t\right)\right]\mathrm{d}t\notag\\
			&=\int_{\mathbb{R}}P\left(x,-y-t\right)\left[\mathscr{A}\left(t\right)+I\mathscr{B}\left(t\right)\right]\mathrm{d}t\\
			&=\Psi\left(x-Iy\right).
		\end{align*}
		By Proposition \ref{prop-extl(f)inN}, for every $J\in\mathbb{S}$, we have that
		\begin{align*}
			f\left(x+Jy\right)
			=&\frac{1}{2}\left[\Psi\left(x+Iy\right)+\Psi\left(x-Iy\right)\right]+\frac{JI}{2}\left[\Psi\left(x-Iy\right)-\Psi\left(x+Iy\right)\right]\notag\\
           =&\int_{\mathbb{R}}P\left(x,y-t\right)\left[\mathscr{A}(t)+J\mathscr{B}(t)\right]\mathrm{d}t
		\end{align*}
        is in $\mathscr{N}_{l}\left(\mathbb{H}_{+}\right)$. Let $F^{J}\left(y\right)=f(Jy)$ be the NTBL of $f$ at $Jy$. Then we have that
		\begin{equation}\label{PW-LpHp-fJ+}
			F^{J}\left(y\right)=\mathscr{A}(y)+J\mathscr{B}(y)
		\end{equation}
		holds almost every $y\in\mathbb{R}$. Thus, by Minkowski's inequality, we have $\lVert F^{J}\rVert_{L^{p}(\mathbb{R})}\leq\lVert \mathscr{A}\rVert_{L^{p}(\mathbb{R})}+\lVert \mathscr{B}\rVert_{L^{p}(\mathbb{R})}<\infty$. Using the arbitrariness of $J$, we have $$\sup_{J\in\mathbb{S}}\lVert F^{J}\rVert_{L^{p}(\mathbb{R})}\leq\lVert \mathscr{A}\rVert_{L^{p}(\mathbb{R})}+\lVert \mathscr{B}\rVert_{L^{p}(\mathbb{R})}<\infty.$$ So, we have $f\in H^{p}\left(\mathbb{H}_{+}\right)$. Moreover, since $f\in\mathscr{N}_{l}(\mathbb{H}_{+})\cap H^{p}(\mathbb{H}_{+})$, by Lemma \ref{PW-NH+-HPF}, we have that for every $J\in\mathbb{S}$ and $x>0,y\in\mathbb{R}$,
  \begin{align*}
      f\left(x+Jy\right)
			=&\frac{1}{\sqrt{2\pi}}\int_{-\infty}^{0}e^{(x+Jy)t}\mathscr{F}_{E}\left(f\right)(t)\mathrm{d}t.
  \end{align*}
  In particular,  $\mathscr{F}_{E}\left(f\right)=\mathscr{F}_{J}(\mathscr{A}+J\mathscr{B})$.
  Furthermore, using Proposition \ref{prop-RF} and \eqref{PW-LpHp-fJ+}, $$\mathscr{A}(y)=\frac{1}{2}\left[F^{J}(y)+F^{J}(-y)\right]=\alpha(0,y)$$ and $$\mathscr{B}(y)=\frac{J}{2}\left[F^{J}(-y)-F^{J}(y)\right]=\beta(0,y)$$ holds a.e. $y\in\mathbb{R}$.
	\end{proof}
	\begin{proof}[Proof of Theorem \ref{PW-H+H-HPSFT}]
		By $f\in\mathscr{R}_{l}\left(\mathbb{H}_{+}\right)$ and Proposition \ref{prop-NE}, for every $\tilde{I}\in\mathbb{S}$, $x>0$ and $y\in\mathbb{R}$, we have
		\begin{align}\label{PW-H+H-extf}
			f\left(x+\tilde{I}y\right)=h_0\left(x+\tilde{I}y\right)+h_1\left(x+\tilde{I}y\right)I+h_2\left(x+\tilde{I}y\right)J+h_3\left(x+\tilde{I}y\right)K,
		\end{align}
		where $h_{m}\in\mathscr{N}_{l}\left(\mathbb{H}_{+}\right)$, $0\leq m\leq3$ and $K=IJ$. Similarly, for $I'\in\mathbb{S}$, we get $\Big|h_{m}\left(x+\tilde{I}y\right)\Big|\leq \Big\lvert f\left(x+I'y\right)\Big\rvert+\Big\lvert f\left(x-I'y\right)\Big\rvert$, $0\leq m\leq3$. Since $f\in H^{p}\left(\mathbb{H}_{+}\right)$, we have $h_{m}\in H^{p}\left(\mathbb{H}_{+}\right)$, $0\leq m\leq3$. Using Lemma \ref{PW-NH+-HPF}, we obtain that $\mathscr{F}_{E}\left(h_{m}\right)$ is real-valued, $\operatorname{supp}\mathscr{F}_{E}\left(h_{m}\right)\subset\left(-\infty,0\right]$ and
		\begin{equation}\label{PW-H+H-hl}
			h_{m}\left(x+\tilde{I}y\right)=\frac{1}{\sqrt{2\pi}}\int_{-\infty}^{0}e^{\left(x+\tilde{I}y\right)t}\mathscr{F}_{E}\left(h_{m}\right)\left(t\right)\mathrm{d}t,m=0,1,2,3.
		\end{equation}
		Combining (\ref{PW-H+H-extf}) and (\ref{PW-H+H-hl}), we have
		\begin{align*}
			&f\left(x+\tilde{I}y\right)\\
			=&h_0\left(x+\tilde{I}y\right)+h_1\left(x+\tilde{I}y\right)I+h_2\left(x+\tilde{I}y\right)J+h_3\left(x+\tilde{I}y\right)K\\
			=&\frac{1}{\sqrt{2\pi}}\int_{-\infty}^{0}e^{\left(x+\tilde{I}y\right)t}\mathscr{F}_{E}\left(h_{0}\right)\left(t\right)\mathrm{d}t+\frac{1}{\sqrt{2\pi}}\int_{-\infty}^{0}e^{\left(x+\tilde{I}y\right)t}\mathscr{F}_{E}\left(h_{1}\right)\left(t\right)\mathrm{d}tI\\
			&+\frac{1}{\sqrt{2\pi}}\int_{-\infty}^{0}e^{\left(x+\tilde{I}y\right)t}\mathscr{F}_{E}\left(h_{2}\right)\left(t\right)\mathrm{d}tJ+\frac{1}{\sqrt{2\pi}}\int_{-\infty}^{0}e^{\left(x+\tilde{I}y\right)t}\mathscr{F}_{E}\left(h_{3}\right)\left(t\right)\mathrm{d}tK\\
			=&\frac{1}{\sqrt{2\pi}}\int_{-\infty}^{0}e^{\left(x+\tilde{I}y\right)t}\left[\mathscr{F}_{E}\left(h_{0}\right)\left(t\right)+\mathscr{F}_{E}\left(h_{1}\right)\left(t\right)I+\mathscr{F}_{E}\left(h_{2}\right)\left(t\right)J+\mathscr{F}_{E}\left(h_{3}\right)\left(t\right)K\right]\mathrm{d}t\\
			=&\frac{1}{\sqrt{2\pi}}\int_{-\infty}^{0}e^{\left(x+\tilde{I}y\right)t}\mathscr{F}_{E}\left(f\right)\left(t\right)\mathrm{d}t
		\end{align*}
		and $\lvert\mathscr{F}_{E}\left(f\right)\left(t\right)\rvert^{2}=\sum_{m=0}^{3}\lvert\mathscr{F}_{E}\left(h_{m}\right)\left(t\right)\rvert^{2}$. It is obvious that $$\operatorname{supp}\mathscr{F}_{E}\left(f\right)\subset\bigcup_{m=0}^{3}\operatorname{supp}\mathscr{F}_{E}\left(h_{m}\right)\subset\left(-\infty,0\right].$$
	\end{proof}
	\begin{proof}[Proof of Theorem \ref{PW-H+H-SFTHP}]
		 
   Let $F:\mathbb{R}\to\mathbb{H}$. Then for a fixed $I\in\mathbb{S}$, it can be written as $$F=F_{0}+I F_{1}+J F_{2}+IJ F_{3}=\left(F_{0}+I F_{1}\right)+\left(F_{2}+I F_{3}\right)J$$ where $F_{0},F_{1},F_{2}$ and $F_{3}$ are real-valued, and $J\in\mathbb{S}$ with $J\perp I$. Let $G_{0}:=F_{0}+I F_{1}$ and $G_{1}:=F_{2}+I F_{3}$. It is easy to show that
     \[|F|^{2}=|G_{0}|^{2}+|G_{1}|^{2}\]
     and
     \[|\mathscr{F}_{I}\left(F\right)|^{2}=|\mathscr{F}_{I}\left(G_{0}\right)|^{2}+|\mathscr{F}_{I}\left(G_{1}\right)|^{2}.\]
     Since $F\in L^{p}(\mathbb{R}), 1\leq p\leq2$ and $\operatorname{supp}\mathscr{F}_{I}\left(F\right)\subset\left(-\infty,0\right]$, then we have $G_{0},G_{1}\in L^{p}(\mathbb{R})$ and $\operatorname{supp}\mathscr{F}_{I}\left(G_{0}\right),\operatorname{supp}\mathscr{F}_{I}\left(G_{1}\right)\subset\left(-\infty,0\right]$. Using Lemma \ref{PW-PICI-Lp}, we have that
    \begin{equation*}
			\Psi_{0}\left(x+Iy\right)=\int_{\mathbb{R}}P\left(x,y-t\right)G_{0}\left(t\right)\mathrm{d}t
		\end{equation*} 
  and
  \begin{equation*}
      \Psi_{1}\left(x+Iy\right)=\int_{\mathbb{R}}P\left(x,y-t\right)G_{1}\left(t\right)\mathrm{d}t
  \end{equation*}
  are in $H^{p}(\Pi_{+,I})$ where $G_{0}\left(y\right)$ and  $G_{1}\left(y\right)$ are the NTBL of $\Psi_{0}$ and $\Psi_{1}$ at $Iy$, respectively.
  Let $\Psi=\Psi_{0}+\Psi_{1} J$. Then 
  \[\Psi\left(x+Iy\right)=\int_{\mathbb{R}}P\left(x,y-t\right)F\left(t\right)\mathrm{d}t\]
  is in $H^{p}(\Pi_{+,I})$ where $F(y)$ is the NTBL of $\Psi$ at $Iy$.
  By Proposition \ref{prop-SRE}, for every $\widetilde{I}\in\mathbb{S}$, we have 
		\begin{align*}
			f\left(x+\widetilde{I}y\right):=&\ \mathbf{ext}_{l}(\Psi)\left(x+\widetilde{I}y\right)\notag\\
			=&\frac{1}{2}\left[\Psi\left(x-Iy\right)+\Psi\left(x+Iy\right)\right]+\frac{\widetilde{I}I}{2}\left[\Psi\left(x-Iy\right)-\Psi\left(x+Iy\right)\right]\notag\\
			=&\frac{1}{2}\left[\int_{\mathbb{R}}P\left(x,y-t\right)F\left(-t\right)\mathrm{d}t+\int_{\mathbb{R}}P\left(x,y-t\right)F\left(t\right)\mathrm{d}t\right]\notag\\
			&+\frac{\widetilde{I}I}{2}\left[\int_{\mathbb{R}}P\left(x,y-t\right)F\left(-t\right)\mathrm{d}t-\int_{\mathbb{R}}P\left(x,y-t\right)F\left(t\right)\mathrm{d}t\right]\notag\\
			=&\int_{\mathbb{R}}P\left(x,y-t\right)\left(\frac{1}{2}\left[F\left(-t\right)+F\left(t\right)\right]+\frac{\widetilde{I}I}{2}\left[F\left(-t\right)-F\left(t\right)\right]\right)\mathrm{d}t.
		\end{align*}
		It follows that $f\left(x+\widetilde{I}y\right)$ is the left slice regular function on $\mathbb{H}_{+}$ and when $\widetilde{I}=I$, $f(x+Iy)=\Psi\left(x+Iy\right)$. Let $F^{\widetilde{I}}\left(y\right)=f(\widetilde{I}y)$ be the NTBL of $f$ at $\widetilde{I}y$. Then $$F^{\widetilde{I}}\left(y\right)=\frac{1}{2}\left[F\left(-y\right)+F\left(y\right)\right]+\frac{\widetilde{I}I}{2}\left[F\left(-y\right)-F\left(y\right)\right]$$ holds almost every $y$. It is easy to see that
		\begin{equation*}
			\int_{-\infty}^{\infty}\lvert F^{\widetilde{I}}\left(y\right)\rvert^{p}\mathrm{d}y\leq2^{p}\int_{-\infty}^{\infty}\lvert F\left(y\right)\rvert^{p}\mathrm{d}y<\infty.
		\end{equation*}
		So, we have $\sup_{\widetilde{I}\in\mathbb{S}}\int_{-\infty}^{\infty}\lvert F^{\widetilde{I}}\left(y\right)\rvert^{p}\mathrm{d}y<\infty$. Hence, we have $f\in H^{p}\left(\mathbb{H}_{+}\right)$. Moreover, if $f\in H^{p}\left(\mathbb{H}_{+}\right)$, by Theorem \ref{PW-H+H-HPSFT}, we have that for every $\widetilde{I}\in\mathbb{S}$ and $x>0,y\in\mathbb{R}$,
  \begin{equation*}
      f(x+\widetilde{I}y)=\frac{1}{\sqrt{2\pi}}\int_{-\infty}^{0}e^{(x+\widetilde{I}y)t}\mathscr{F}_{E}\left(f\right)(t)\mathrm{d}t,
  \end{equation*}
  where $\mathscr{F}_{E}\left(f\right)=\mathscr{F}_{I}(F)$.
	\end{proof}

	\section{Whittaker–Kotelnikov–Shannon Sampling Theorem}
	The classical complex Paley-Wiener space, the sampling theorem and the sinc function play important roles in Fourier analysis, numerical analysis and signal analysis (see \cite{higgins1996sampling,lund1992sinc,stenger1993numerical}). So it is the motivation for us to study the sampling theorem for slice regular functions. In the following we first define the Paley-Wiener space $PW_{A}^{2}\left(\mathbb{H}\right)$, $A>0$, where
	\begin{equation*}
		PW_{A}^{2}\left(\mathbb{H}\right)=\left\{f\in\mathscr{R}_{l}\left(\mathbb{H}\right), f|_{\mathbb{R}}\in L^2(\mathbb{R}),\operatorname{supp}\mathscr{F}_{I}\left(f|_{\mathbb{R}}\right)\subset\left[-A,A\right]\text{ for every } I\in\mathbb{S}\right\}.
	\end{equation*}
 Recall that for all $\mathbf{q}\in\mathbb{H}$, the sine function (see \cite{colombo2016entire}, Remark 2.8) is given by
 $$\sin\left(\mathbf{q}\right)=\sum_{n=0}^{\infty}\left(-1\right)^{n}\frac{\mathbf{q}^{2n+1}}{\left(2n+1\right)!}.$$
 
 Now, we give the definition of the sinc function as follows.

\begin{definition}\label{def-sinc}
    For all $\mathbf{q}\in\mathbb{H}$, the sinc function is defined by
    \begin{equation*}
        \sinc(\mathbf{q})=\sum_{k=0}^{\infty}\mathbf{q}^{2k}\frac{(-\pi^2)^k}{(2k+1)!}=(\pi\mathbf{q})^{-1}\sin(\pi\mathbf{q})=\sin(\pi\mathbf{q})(\pi\mathbf{q})^{-1}.
    \end{equation*}
\end{definition}
It follows from Proposition \ref{prop-N-series} that $\sinc(\mathbf{q})\in\mathscr{N}_{l}(\mathbb{H})$.
\begin{remark}
    For all $\mathbf{q}=x+Iy\in\mathbb{H}$, the
 integral expression of sinc is given by
    \begin{equation*}
        \sinc(x+Iy)=\frac{1}{2\pi}\int_{-\pi}^{\pi}e^{I(x+Iy)t}\mathrm{d}t=\frac{\sin\left(\pi(x+Iy)\right)\cdot(x-Iy)}{\pi|x+Iy|^{2}}.
    \end{equation*}
\end{remark}

	Now, we give a reproducing kernel for $PW_A^2(\mathbb{H})$.
	\begin{theorem}[Reproducing kernel formula]
		If $f\in PW_{A}^{2}\left(\mathbb{H}\right)$, then
		\begin{equation*}
			f\left(\mathbf{q}\right)=\frac{A}{\pi}\int_{-\infty}^{
				\infty}\sinc\frac{A}{\pi}\left(\mathbf{q}-t\right)f\left(t\right)\mathrm{d}t,\ \mathbf{q}\in\mathbb{H}.
		\end{equation*}
	\end{theorem}
	\begin{proof}
		By Theorem \ref{PW-QQ2} and the multiplication formula of Fourier transform, for every $\mathbf{q}=x+Jy\in\mathbb{H}$, we have
		\begin{align*}
			f\left(x+Jy\right)&=\frac{1}{\sqrt{2\pi}}\int_{-A}^{A}e^{J\left(x+Jy\right)t}\mathscr{F}_{J}\left(f|_{\mathbb{R}}\right)\left(t\right)\mathrm{d}t\\
			&=\frac{1}{\sqrt{2\pi}}\int_{-\infty}^{\infty}\chi_{\left[-A,A\right]}\left(t\right)e^{J\left(x+Jy\right)t}\mathscr{F}_{J}\left(f|_{\mathbb{R}}\right)\left(t\right)\mathrm{d}t\\
			&=\frac{1}{\sqrt{2\pi}}\int_{-\infty}^{\infty}\mathscr{F}_{J}\left(\chi_{\left[-A,A\right]}\left(\cdot\right)e^{J\left(x+Jy\right)\left(\cdot\right)}\right)\left(t\right)f\left(t\right)\mathrm{d}t\\
			&=\frac{A}{\pi}\int_{-\infty}^{\infty}\sinc\frac{A}{\pi}\left(\left(x+Jy\right)-t\right)f\left(t\right)\mathrm{d}t.
		\end{align*}
	\end{proof}
	\begin{remark}\label{sinc<c}
    Here, the sinc function in Definition \ref{def-sinc} has a property analogous to the complex sinc function, as shown below: For $M>0$, $p>1$ and $\frac{1}{p}+\frac{1}{p'}=1$, we have
		\begin{equation*}
			\sum_{k\in\mathbb{Z}}|\sinc\left(\mathbf{q}-k\right)|^p<p'e^{pM\pi}
		\end{equation*}
        for all $\mathbf{q}\in\Omega_{M}:=\left\{\mathbf{q}=x+Jy\in\mathbb{H}:x\in\mathbb{R},|y|<M\right\}.$ The proof is similar to that given in \cite[Lemma 2.2]{butzer2011sampling}. It is useful to prove Theorem \ref{Sampling theorem}.
	\end{remark}

	\begin{theorem}[Sampling Theorem]\label{Sampling theorem}
		For $f\in PW_{A}^2(\mathbb{H})$, we have
		\begin{equation*}
			f\left(\mathbf{q}\right)=\sum_{k\in\mathbb{Z}}\sinc\left(\frac{A\mathbf{q}}{\pi}-k\right)f\left(\frac{\pi k}{A}\right),\ \mathbf{q}\in\mathbb{H}.
		\end{equation*}
		Moverover, the convergence is absolute and uniform on $\Omega_{M}=\{\mathbf{q}=x+Jy\in\mathbb{H}:x\in\mathbb{R},|y|<M\}$ for $M<\infty$, and on compact sets of $\mathbb{H}$.
	\end{theorem}
	\begin{proof}
		For any fixed $\mathbf{q}=x+Jy\in\mathbb{H}$, expanding the function $e^{J\left(x+Jy\right)t}$ into its uniformly convergent Fourier series in $\left[-A,A\right]$, we have 
		\begin{equation}\label{e^J}
			e^{J\left(x+Jy\right)t}=\sum_{k\in\mathbb{Z}}\sinc\left(\frac{A\left(x+Jy\right)}{\pi}-k\right)e^{J\frac{\pi k}{A}t},  
		\end{equation}
		
		where the terms $\sinc\left(\frac{A\left(x+Jy\right)}{\pi}-k\right)$ are the Fourier coefficients. By Theorem \ref{PW-QQ2}, we have
		\begin{align*}
			f\left(x+Jy\right)&=\frac{1}{\sqrt{2\pi}}\int_{-A}^{A}e^{J\left(x+Jy\right)t}\mathscr{F}_{J}\left(f|_{\mathbb{R}}\right)\left(t\right)\mathrm{d}t\\
			&=\frac{1}{\sqrt{2\pi}}\int_{-A}^{A}\sum_{k\in\mathbb{Z}}\sinc\left(\frac{A\left(x+Jy\right)}{\pi}-k\right)e^{J\frac{\pi k}{A}t}\mathscr{F}_{J}\left(f|_{\mathbb{R}}\right)\left(t\right)\mathrm{d}t\\
			&=\frac{1}{\sqrt{2\pi}}\sum_{k\in\mathbb{Z}}\sinc\left(\frac{A\left(x+Jy\right)}{\pi}-k\right)\int_{-A}^{A}e^{J\frac{\pi k}{A}t}\mathscr{F}_{J}\left(f|_{\mathbb{R}}\right)\left(t\right)\mathrm{d}t\\
			&=\sum_{k\in\mathbb{Z}}\sinc\left(\frac{A\left(x+Jy\right)}{\pi}-k\right)f\left(\frac{\pi k}{A}\right).
		\end{align*}
		The exchange of the summation and the integration is justified by the uniform convergence of \eqref{e^J}. By the Cauchy-Schwarz inequality, we have
		\begin{align}\label{CS}
			&\sum_{k\in\mathbb{Z}}\Big|\sinc\left(\frac{A\left(x+Jy\right)}{\pi}-k\right)f\left(\frac{\pi k}{A}\right)\Big|\notag\\
			\leq&\left(\sum_{k\in\mathbb{Z}}\Big|\sinc\left(\frac{A\left(x+Jy\right)}{\pi}-k\right)\Big|^2\right)^{\frac{1}{2}}\left(\sum_{k\in\mathbb{Z}}\Big|f\left(\frac{\pi k}{A}\right)\Big|^2\right)^{\frac{1}{2}}.
		\end{align}
		Parseval's identity gives
		\begin{equation*}
			\int_{-A}^{A}|\mathscr{F}_{J}\left(f|_{\mathbb{R}}\right)\left(t\right)|^{2}\mathrm{d}t=2A\sum_{k\in\mathbb{Z}}|c_k|^2=\frac{\pi}{A}\sum_{k\in\mathbb{Z}}\Big|f\left(\frac{\pi k}{A}\right)\Big|^2,
		\end{equation*}
		where $$c_k=\frac{1}{2A}\int_{-A}^{A}e^{-J\frac{\pi k}{A}x}\mathscr{F}_{J}\left(f|_{\mathbb{R}}\right)\left(x\right)\mathrm{d}x.$$ 
		Using the Plancherel theorem 
		and $\operatorname{supp}\mathscr{F}_{J}\subset\left[-A,A\right]$, we have
		\begin{equation*}
			\int_{-A}^{A}|\mathscr{F}_{J}\left(f|_{\mathbb{R}}\right)\left(t\right)|^{2}\mathrm{d}t=\int_{\mathbb{R}}|\mathscr{F}_{J}\left(f|_{\mathbb{R}}\right)\left(t\right)|^{2}\mathrm{d}t=\int_{\mathbb{R}}|f\left(x\right)|^2\mathrm{d}x.
		\end{equation*}
		It follows that \begin{equation}\label{fFS}
			\frac{\pi}{A}\sum_{k\in\mathbb{Z}}\Big|f\left(\frac{\pi k}{A}\right)\Big|^2=\int_{\mathbb{R}}|f\left(x\right)|^2\mathrm{d}x<\infty.
		\end{equation}
		
		Combining \eqref{CS}, \eqref{fFS} and Remark \ref{sinc<c}, we conclude that the sampling series is absolutely and uniformly convergent on $\Omega_{M}=\left\{\mathbf{q}=x+Jy\in\mathbb{H}:x\in\mathbb{R},|y|<M\right\}$ for $M<\infty$.
	\end{proof}
	\begin{prop}
		For $f\in PW_{A}^2(\mathbb{H})$, we have, for every  $J\in\mathbb{S}$ and all $x,y\in\mathbb{R}$,
		\begin{equation*}
			f\left(x+Jy\right)=\sum_{k\in\mathbb{Z}}\sinc\left(\frac{A(x+Jy)}{\pi}-k\right)f\left(\frac{\pi k}{A}\right)
		\end{equation*}
		converges in the $L^2(\mathbb{R})$-norm.
	\end{prop}
	\begin{proof}
		For $x\in\mathbb{R}$, we have
		\begin{align*}
			\mathscr{F}_{J}\left(\chi_{[-A,A]}\left(t\right)e^{Jxt}\right)\left(\xi\right)=\frac{1}{\sqrt{2\pi}}\int_{-A}^{A}e^{-J\xi t}e^{Jxt}\mathrm{d}t=\frac{2A}{\sqrt{2\pi}}\sinc\left(\frac{A}{\pi}\left(\xi-x\right)\right),
		\end{align*}
		where $\chi_{[-A,A]}$ is the characteristic function on $[-A,A]$.
		By $$\mathscr{F}_{J}\left(\chi_{[-A,A]}\left(t\right)e^{Jxt}\right)\left(\xi\right)=\mathscr{F}_{J}^{-1}\left(\chi_{[-A,A]}\left(t\right)e^{-Jxt}\right)\left(\xi\right)$$ and $$\mathscr{F}_{J}\left(\mathscr{F}_{J}^{-1}\left(F\right)\right)=F,$$
		we have $\mathscr{F}_{J}\left(\sinc\left(\frac{A}{\pi}\left(\cdot-x\right)\right)\right)\left(t\right)=\frac{\sqrt{2\pi}}{2A}\chi_{[-A,A]}\left(t\right)e^{-Jxt}$.
		
		For $f\in PW_{A}^{2}(\mathbb{H})$, using the Plancherel theorem and Theorem \ref{PW-QQ2}, we have
		\begin{align}\label{L2norm}
			&\Big\Vert f\left(t\right)-\sum_{|k|\leq n}\sinc\left(\frac{At}{\pi}-k\right)f\left(\frac{\pi k}{A}\right)\Big\Vert_{L^2\left(\mathbb{R}\right)}\notag\\
			=&\Big\Vert \mathscr{F}_{J}\left(f|_{\mathbb{R}}\right)\left(t\right)-\left\{\frac{\sqrt{2\pi}}{2A}\sum_{|k|\leq n}e^{-J\frac{\pi k}{A}t}f\left(\frac{\pi k}{A}\right)\right\}\chi_{[-A,A]}\left(t\right)\Big\Vert_{L^2\left(\mathbb{R}\right)}\notag\\
			=&\Big\Vert \mathscr{F}_{J}\left(f|_{\mathbb{R}}\right)\left(t\right)-\frac{\sqrt{2\pi}}{2A}\sum_{|k|\leq n}e^{-J\frac{\pi k}{A}t}f\left(\frac{\pi k}{A}\right)\Big\Vert_{L^2\left(-A,A\right)}\notag\\
			=&\Big\Vert \mathscr{F}_{J}\left(f|_{\mathbb{R}}\right)\left(t\right)-\frac{\sqrt{2\pi}}{2A}\sum_{|k|\leq n}e^{J\frac{\pi k}{A}t}f\left(-\frac{\pi k}{A}\right)\Big\Vert_{L^2\left(-A,A\right)}.
		\end{align}
        Using Proposition \ref{prop-NE}, we have $f=f_{0}+f_{1}\tilde{I}+f_{2}\tilde{J}+f_{3}\tilde{I}\tilde{J}$ where $f_{m}\in\mathscr{N}_{l},m=0,1,2,3$ and $\tilde{I},\tilde{J}$ are mutually orthogonal quaternionic imaginary units. It follows that 
       \begin{equation*}
           \begin{split}
               &\Big\Vert \mathscr{F}_{J}\left(f|_{\mathbb{R}}\right)\left(t\right)-\frac{\sqrt{2\pi}}{2A}\sum_{|k|\leq n}e^{J\frac{\pi k}{A}t}f\left(-\frac{\pi k}{A}\right)\Big\Vert_{L^2\left(-A,A\right)}\\
               =&\Big\Vert\mathscr{F}_{J}(f_{0}|_{\mathbb{R}})(t)-\frac{\sqrt{2\pi}}{2A}\sum_{|k|\leq n}e^{J\frac{\pi k}{A}t}f_{0}\left(-\frac{\pi k}{A}\right)\\
               &+\mathscr{F}_{J}(f_{1}|_{\mathbb{R}})(t)\tilde{I}-\frac{\sqrt{2\pi}}{2A}\sum_{|k|\leq n}e^{J\frac{\pi k}{A}t}f_{1}\left(-\frac{\pi k}{A}\right)\tilde{I}\\
               &+\mathscr{F}_{J}(f_{2}|_{\mathbb{R}})(t)\tilde{I}-\frac{\sqrt{2\pi}}{2A}\sum_{|k|\leq n}e^{J\frac{\pi k}{A}t}f_{2}\left(-\frac{\pi k}{A}\right)\tilde{J}\\
               &+\mathscr{F}_{J}(f_{3}|_{\mathbb{R}})(t)\tilde{I}\tilde{J}-\frac{\sqrt{2\pi}}{2A}\sum_{|k|\leq n}e^{J\frac{\pi k}{A}t}f_{3}\left(-\frac{\pi k}{A}\right)\tilde{I}\tilde{J}\Big\Vert_{L^2\left(-A,A\right)}\\
               \leq&\sum_{m=0}^{3}\Big\Vert \mathscr{F}_{J}\left(f_{m}|_{\mathbb{R}}\right)\left(t\right)-\frac{\sqrt{2\pi}}{2A}\sum_{|k|\leq n}e^{J\frac{\pi k}{A}t}f_{m}\left(-\frac{\pi k}{A}\right)\Big\Vert_{L^2\left(-A,A\right)}.
           \end{split}
       \end{equation*}
        Since $f\in PW_{A}^{2}(\mathbb{H})$, we have $f_{m}\in PW_{A}(\mathbb{H})$. Note that
       \begin{equation*}
			\frac{\sqrt{2\pi}}{2A}f_{m}\left(-\frac{\pi k}{A}\right)=\frac{1}{2A}\int_{-A}^{A}e^{-J\frac{\pi k}{A}t}\mathscr{F}_{J}\left(f_{m}|_{\mathbb{R}}\right)\left(t\right)\mathrm{d}t.
		\end{equation*}
       Thus, $\frac{\sqrt{2\pi}}{2A}\sum_{k\in\mathbb{Z}}e^{J\frac{\pi k}{A}t}f_{m}\left(-\frac{\pi k}{A}\right)$ is the Fourier series of $\mathscr{F}_{J}\left(f_{m}|_{\mathbb{R}}\right)\in L^{2}\left(-A,A\right)$. It is easy to see that 
       \[\lim_{n\to\infty}\Big\Vert \mathscr{F}_{J}\left(f_{m}|_{\mathbb{R}}\right)\left(t\right)-\frac{\sqrt{2\pi}}{2A}\sum_{|k|\leq n}e^{J\frac{\pi k}{A}t}f_{m}\left(-\frac{\pi k}{A}\right)\Big\Vert_{L^2\left(-A,A\right)}=0.\]
        Hence  the sampling series converges to $f$ in the $L^2(\mathbb{R})$-norm.
	\end{proof}
    
	\section*{Acknowledgements.}
    We would like to express our deep gratitude to the referee for his/her thorough reading, useful comments and suggestions, which greatly helps to improve the presentation and quality of the paper.
    
	W. X. Mai was supported by 
	the Science and Technology Development Fund, Macau SAR (No. 0133/2022/A, 0022/2023/ITP1).
	P. Dang was supported by FRG Program of the Macau University of Science
	and Technology, No. FRG-23-033-FIE, and the Science and Technology Development Fund, Macau SAR (No. 0030/2023/ITP1).

\end{document}